\documentclass[a4paper,11pt]{amsart}

\usepackage{amsmath}
\usepackage{amsfonts}
\usepackage{amssymb}
\usepackage{amsthm}
\usepackage{enumerate}

\usepackage[square,numbers]{natbib}

\newtheorem{theorem}{Theorem}
\newtheorem{claim}[theorem]{Claim}

\newtheorem{lemma}[theorem]{Lemma}
\newtheorem{proposition}[theorem]{Proposition}

\theoremstyle{definition}
\newtheorem{definition}[theorem]{Definition}
\newtheorem{assumptions}[theorem]{Assumptions}

\theoremstyle{remark}
\newtheorem{remark}[theorem]{Remark}

\newtheorem{case}{Case}

\numberwithin{theorem}{section}
\numberwithin{equation}{section}

\newcommand{\dd}{\; \mathrm{d}}

\def\Xint#1{\mathchoice
{\XXint\displaystyle\textstyle{#1}}
{\XXint\textstyle\scriptstyle{#1}}
{\XXint\scriptstyle\scriptscriptstyle{#1}}
{\XXint\scriptscriptstyle\scriptscriptstyle{#1}}
\!\int}
\def\XXint#1#2#3{{\setbox0=\hbox{$#1{#2#3}{\int}$}
\vcenter{\hbox{$#2#3$}}\kern-.5\wd0}}

\def\dashint{\Xint-}

\begin{document}
\title[Differentiability in a Null Set]{Differentiability of Lipschitz Functions in Lebesgue Null Sets}
\author{David Preiss}
\author{Gareth Speight}

\thanks{The research leading to these results has received funding from the European Research Council under the European Union's Seventh Framework Programme (FP/2007-2013) / ERC Grant Agreement n.2011-ADG-20110209. The second author was supported by EPSRC funding.}

\begin{abstract}
We show that if $n>1$ then there exists a Lebesgue null set in $\mathbb{R}^{n}$ containing a point of differentiability of each Lipschitz function $f\colon \mathbb{R}^{n} \to \mathbb{R}^{n-1}$; in combination with the work of others, this completes the investigation of when the classical Rademacher theorem admits a converse. Avoidance of $\sigma$-porous sets, arising as irregular points of Lipschitz functions, plays a key role in the proof.
\end{abstract}

\maketitle

\section{Introduction}

Rademacher's theorem that Lipschitz functions on $\mathbb{R}^n$ are differentiable almost everywhere is intrinsically important and has been the source of many modern developments. From the developments that are not directly related to the present work, we feel we just have to mention, at least briefly, the work of \citet{Che99}, \citet{Kei04a} and \citet{Bat13}, which starting from the notion of metric measure spaces satisfying the Poincar\'e inequality eventually led to understanding differentiability of Lipschitz functions on metric spaces in a way similar to Rademacher's theorem. On the other hand, the investigation of validity of an infinite dimensional generalization of Rademacher's theorem, a survey of older results in Chapters 4--6 of the authoritative book by \citet{BL99}, and very recent progress presented in the recent research monograph by \citet*{LPT12} has been basic for much of what we do here.

To introduce our result, the most natural formulation of the classical Rademacher theorem states that if a Lipschitz function $f \colon \mathbb{R}^n \to \mathbb{R}^m$ is differentiable at no point of a set $A\subset\mathbb{R}^n$, then $A$ must be Lebesgue null. The natural converse of this statement asks: given a Lebesgue null set $A \subset \mathbb{R}^n$, does there exist a Lipschitz function $f\colon \mathbb{R}^n \to \mathbb{R}^m$ which is differentiable at no point of $A$? The answer has been long known to be positive and relatively easy in the case $m=n=1$, when the statement of Rademacher's theorem is a special case of Lebesgue's differentiation of monotone functions. Although we are unable to find the first reference to this result, we may refer the reader to the full description of sets of non-differentiability of real valued Lipschitz functions on the real line due to \citet{Zah46}, or to a modern variant of Zahorski's argument in \cite{FP09}.

The discovery that the converse to the higher dimensional version of Rademacher's theorem is not straightforward came originally as a byproduct of an infinite dimensional differentiability result of \citet{Pre90}: for $n>1$ there is a Lebesgue null set in $\mathbb{R}^n$ containing a point of differentiability for every real valued Lipschitz function: for example, any Lebesgue null $G_\delta$ set containing all lines passing through distinct points with rational coordinates has this property. It seems probable that, similarly to what happened with \cite{Pre90}, a modification of the proof of \cite{LPT12} of a differentiability result for Lipschitz maps of infinite dimensional Hilbert spaces to $\mathbb{R}^2$ would lead to showing that the converse to Rademacher's theorem fails for maps from $\mathbb{R}^n$ to $\mathbb{R}^2$ for $n >2$. However, these authors also show that the corresponding differentiability result for $\mathbb{R}^3$ valued maps is false, thereby indicating that for general $m,n$ new methods are needed. 

In combination with two recently announced developments, by \citet*{ACP10} and by \citet*{CJ13}, our result completely answers the question of validity of the converse to the Rademacher theorem: it holds if and only if $m\ge n$. Indeed, \cite{ACP10} shows this when $n=2$ and, for general $n$, provides necessary and sufficient geometric criteria for a set to be contained in the non-differentiability set of a Lipschitz function $f \colon \mathbb{R}^n \to \mathbb{R}^n$. The sets satisfying these criteria form a $\sigma$-ideal, implying that for any given $n$ the problem of validity of the converse to the Rademacher theorem for functions $f \colon \mathbb{R}^n \to \mathbb{R}^m$ has the same answer for $m\ge n$. The question whether this $\sigma$-ideal coincides with the $\sigma$-ideal of Lebesgue null sets was open until \citet{CJ13} announced a very deep and difficult result showing that this is indeed the case. Together, these results imply that the converse to the Rademacher theorem is true provided that $m\ge n$. Here we fill in the last piece of the puzzle by showing that in all remaining cases the converse actually fails.

\begin{theorem}\label{theoremdiffinnullset}
Suppose $n>1$. Then there exists a Lebesgue null set $N\subset \mathbb{R}^{n}$ containing a point of differentiability for every Lipschitz function $f\colon \mathbb{R}^{n} \to \mathbb{R}^{n-1}$.
\end{theorem}

Perhaps surprisingly, our approach to proving this is not related to known results on $\varepsilon$-differentiability. The notion of $\varepsilon$-differentiability is defined similarly to differentiability but with fixed error $\varepsilon$ in the first order approximation of a function by its derivative (rather than arbitrarily small error on sufficiently small scales). The $\varepsilon$-differentiability results appeared first in the infinite dimensional context in \cite{LP96} (see Chapter 4 of \cite{LPT12} for further developments). The restriction of their proof to the finite dimensional situation would provide a Lebesgue null set $N \subset \mathbb{R}^n$ such that every Lipschitz map $f \colon \mathbb{R}^n \to \mathbb{R}^{n-2}$ has points of $\varepsilon$-differentiability inside $N$ for every $\varepsilon > 0$. This was improved by \citet{dPH02}: for $n > 1$ there exists a Lebesgue null set $N \subset \mathbb{R}^n$ such that every Lipschitz map $f \colon \mathbb{R}^n \to \mathbb{R}^{n-1}$ has points of $\varepsilon$-differentiability inside $N$ for every $\varepsilon > 0$. The reason why the present development does not build on $\varepsilon$-differentiability results is treated in detail in the infinite dimensional context throughout \cite{LPT12}. Here we just mention that, while $\varepsilon$-differentiability has been proved for Lipschitz maps of a Hilbert space to any $\mathbb{R}^n$, differentiability is known only for Lipschitz maps to $\mathbb{R}$ or $\mathbb{R}^2$.

The direction our approach to proving Theorem \ref{theoremdiffinnullset} took started with the second named author's answer in \cite{Spe13} to the question from \cite{LPT12} on size of $\sigma$-porous sets. These sets (defined in Definition \ref{porousdefinition}) form a subclass of non-differentiability sets, and one of the contributions of \cite{LPT12} was in better understanding of what was noticed in \cite{LP03}, that knowledge of smallness of porous sets is an important step in proving a differentiability result. Results of Chapter 10 in \cite{LPT12} show that in many spaces, including the finite dimensional ones, $\sigma$-porous sets are null on typical curves as well as on typical 2-dimensional surfaces (where `typical' is understood in the sense of Baire category). The result of \citet{Spe13} shows that even in $\mathbb{R}^4$ this is no longer the case for 3-dimensional surfaces. This opened the door to questions whether the deep infinite dimensional counterexamples from Chapter 14 of \cite{LPT12} have an analogy in the finite dimensional situation. By discovering that the answer is no, the second named author made an important, and as it turned out decisive, step toward the proof of Theorem \ref{theoremdiffinnullset}.

It is natural to ask how small can be sets $N\subset\mathbb{R}^n$ inside which one may find a point of differentiability of every Lipschitz $f \colon \mathbb{R}^n \to \mathbb{R}^{m}$. For $n>m=1$, 
this was studied by Dor\'e and Maleva \cite{DM11} who found such sets can be made compact and of Hausdorff dimension one, and recently Dymond and Maleva \cite{DyM14} proved that they can make them of Minkowski dimension one. It seems possible that the finite dimensional analogy of the above mentioned $\mathbb{R}^2$ valued differentiability result of \cite{LPT12} may lead to similar improvements for $n>m=2$. However, our method differs significantly from the methods used to prove these improvements; so at the present time we can only notice that a simple modification of our arguments (explained in Remark \ref{finrem}) provides, for any $n>m$ and $\tau>0$, a set $N\subset\mathbb{R}^n$ of Hausdorff dimension at most $m+\tau$ containing a point of differentiability of every Lipschitz $f \colon \mathbb{R}^n \to \mathbb{R}^{m}$.

We now fix $n>1$ for the remainder of the paper and briefly describe the structure of the proof of Theorem \ref{theoremdiffinnullset}. It follows that of a proof by \citet*[Theorem 13.1.1]{LPT12}, which constructs points of Fr\'{e}chet differentiability for vector valued Lipschitz functions on infinite dimensional Banach spaces satisfying particular smoothness assumptions.

The idea behind both of the proofs is to apply a variational principle (Lemma \ref{vpstatement}) so that a perturbation of the function
\[(x,T) \mapsto \|f'(x;T)\|_{H}^{2}\]
attains a maximum at some $(x_{\infty},T_{\infty})$. Here $(x,T)$ are pairs such that $f$ is regularly differentiable at $x$ in the direction $T$ (Definition \ref{regularlydifferentiable}), $f'(x;T)$ is the directional derivative of $f$ at $x$ in direction $T$, and $\|\cdot\|_{H}$ is the Hilbert-Schmidt norm. Intuitively, $f$ is regularly differentiable at $x$ in direction $T$ if changes in values of $f$ are approximated by $f'(x;T)$ not only on planes with direction $T$ passing through $x$ but also on nearby parallel planes.

One assumes $f$ is not differentiable at $x_{\infty}$ and uses regular differentiablity to find a pair $(x,T)$ such that $f$ is regularly differentiable at $x$ in direction~$T$ and $\|f'(x;T)\|_{H}^{2}$ is larger than $\|f'(x_{\infty},T_{\infty})\|_{H}^{2}$. Since only a perturbation of the directional derivative was maximized, to get a contradiction, one must also choose $(x,T)$ so that the corresponding change in the perturbing functions is small relative to $\|f'(x;T)\|_{H}^{2}-\|f'(x_{\infty};T_{\infty})\|_{H}^{2}$.

To prove Theorem \ref{theoremdiffinnullset} we consider points $x$ in a Lebesgue null set and need to find points of regular differentiability in the direction of $n-1$ dimensional planes. To do this we use the fact that points where $f$ is differentiable, but not regularly differentiable, in some direction are irregular (Definition \ref{regularirregular} and Lemma \ref{regularpointdiff}). Further, the irregular points of $f$ form a $\sigma$-porous set (Lemma \ref{irregularporous}).

Intuitively, a set is porous (Definition \ref{porousdefinition}) if each point of the set sees relatively large holes in the set on arbitrarily small scales. A set is $\sigma$-porous if it is a countable union of porous sets. The collection of $\sigma$-porous sets in $\mathbb{R}^{n}$ is strictly contained in the family of meager Lebesgue null sets. For a survey of porous sets, including their interesting applications to differentiability, see \citep{Zaj88} and \citep{Zaj05}.

It follows from work of \citet{Tuk89} that there exists a doubling measure (Definition \ref{definitiondoublingmeasure}) $\mu$ on $\mathbb{R}$ which gives full measure to a Lebesgue null set $\widetilde{N} \subset \mathbb{R}$ (Theorem \ref{singulardoublingmeasure}). Since porous sets have measure zero with respect to doubling measures (Proposition \ref{doublingporousnull}) and $\mathcal{L}^{n-1} \times \mu$ is doubling, we can find many points in the Lebesgue null set $\mathbb{R}^{n-1} \times \widetilde{N}$ outside a given $\sigma$-porous set. The Lebesgue null set in Theorem \ref{theoremdiffinnullset} is constructed using affine copies of the set $\mathbb{R}^{n-1} \times \widetilde{N}$ (Definition \ref{nullset}). Since irregular points of $f$ form a $\sigma$-porous set, it follows that such a set contains many points of regular differentiability of $f$ in the direction of $n-1$ dimensional planes. The fact that the measure $\mu$ is doubling also helps us to control the perturbation terms.

\section{Doubling Measures}

We first recall, and for completeness prove, some basic facts about doubling measures.

\begin{definition}\label{definitiondoublingmeasure}
A Borel measure $\mu$ on a metric space $M$ is \emph{doubling} if balls have finite positive measure and there exists $C\geq 1$ such that
\[\mu(B(x,2r))\leq C\mu(B(x,r))\]
for all $x \in M$ and $r>0$.
\end{definition}

Doubling measures give relatively large measure to subballs of relatively large radius.

\begin{lemma}\label{doublingsubballs}
Suppose $\mu$ is a doubling measure on $\mathbb{R}^{m}$. Then there exists a constant $C(\mu)\geq 1$ such that
\[\frac{\mu(B(y,s))}{\mu(B(x,r))} \leq C(\mu) \left(\frac{s}{r}\right)^{C(\mu)}\]
whenever $B(x,r) \subset B(y,s)$.
\end{lemma}

\begin{proof}
Let $N$ be an integer such that $\log_{2}(s/r)+1 \leq N \leq \log_{2}(s/r)+2$. Then $2^{N}r \geq 2s$ and so $B(y,s) \subset B(x,2s) \subset B(x,2^{N}r)$. Hence,
\[\frac{\mu(B(y,s))}{\mu(B(x,r))} \leq \frac{\mu(B(x,2^{N}r))}{\mu(B(x,r))} \leq C(\mu)^{N} \leq C(\mu)^{\log_{2}(s/r)+2} \leq C(\mu)\left( \frac{s}{r} \right)^{C(\mu)}.\]
\end{proof}

We also need to know that doubling measures give small measures to thin annuli, independently of the centre and radius of the associated ball.

\begin{proposition}\label{doublingfills}
Suppose $\mu$ is a doubling measure on $\mathbb{R}^{m}$ and $\varepsilon > 0$. Then there is $\delta>0$ such that
\[\frac{\mu(B(x,(1-t)r))}{\mu(B(x,r))}\geq 1-\varepsilon\]
whenever $x \in \mathbb{R}^{m}$, $r>0$ and $0<t<\delta$.
\end{proposition}

\begin{proof}
Suppose $x \in \mathbb{R}^{m}$ and $r>0$ with
\[\mu(B(x,r)\setminus B(x,(1-2^{-N})r))>\alpha\mu(B(x,r))\]
for some $N \in \mathbb{N}$ and $\alpha>0$. Define, for $i=1, \ldots, N$,
\[A_{i}=B(x,(1-2^{-i-1})r)\setminus B(x,(1-2^{-i})r).\]

Fix $1 \leq i \leq N$. Choose a (necessarily finite) collection $B(x_{j},2^{-i-2}r) \subset A_{i}$ of disjoint open balls such that, for a constant $C_1>0$ depending only on $m$,
\[B(x,r)\setminus B(x,(1-2^{-N})r) \subset \bigcup\nolimits_{j}B(x_{j},C_12^{-i-2}r).\]
Since the balls $B(x_{j},2^{-i-2}r)$ are disjoint it follows from Lemma \ref{doublingsubballs} that, for some constant $C$ depending on $C_1$ and the doubling constant of $\mu$,
\[\mu(B(x,r)\setminus B(x,(1-2^{-N})r)) \leq C\sum\nolimits_j \mu(B(x_{j},2^{-i-1}r)) \leq C\mu(A_{i}).\]
Since the sets $A_{i}$ are disjoint we deduce,
\[\mu(B(x,r))\geq \sum_{i=1}^{N} \mu(A_{i}) \geq N\alpha \mu(B(x,r))/C.\]
Hence by choosing $N$ sufficiently large (independently of $x$ and $r$) we can ensure $\alpha$ is small. This gives the desired conclusion.
\end{proof}

\begin{definition}
If $\mu$ is a doubling measure on $\mathbb{R}^{m}$ and $f\colon \mathbb{R}^{m} \to \mathbb{R}$ is locally integrable we define the \emph{maximal operator} $M_{\mu}f$ by
\[M_{\mu}f(x)=\sup_{r>0} \frac{1}{\mu(B(x,r))}\int_{B(x,r)}|f|\dd\mu.\]
\end{definition}

We recall that if $\mu$ is doubling then the maximal operator cannot increase the norm on $L^{2}(\mu)$ too much \citep[Theorem 2.2]{Hei01}.

\begin{lemma}\label{maxoperator}
Suppose $\mu$ is a doubling measure on $\mathbb{R}^{m}$ and $f \in L^{2}(\mu)$. Then there is a constant $C(\mu) \geq 1$ such that
\[\|M_{\mu}f\|_{L^{2}(\mu)}\leq C(\mu) \|f\|_{L^{2}(\mu)}.\]
\end{lemma}

\section{Regularity Defect}

Recall that a function $f\colon M \to N$ between metric spaces is \emph{Lipschitz} if there is a constant $L \geq 0$ such that
\[d(f(x),f(y))\leq Ld(x,y)\]
for all $x, y \in M$. We denote the smallest such $L$, called the \emph{Lipschitz constant} of $f$, by $\mathrm{Lip}(f)$.

We now discuss regularity and derive estimates needed for the proof of Theorem \ref{theoremdiffinnullset}. Note we identify $\mathbb{R}^{n-1}$ with the subspace of $\mathbb{R}^{n}$ in which the final coordinate is zero.

\begin{definition}\label{regularitydefect}
Let $f\colon \mathbb{R}^{n} \to \mathbb{R}^{n-1}$ be a Lipschitz function and $u \in \mathbb{R}^{n}$. We define the \emph{defect of regularity} of $f$ at the point $u$ by
\[\mathrm{reg}_{n-1}f(u)=\sup_{\substack{|v|+|w|>0\\ v-w \in \mathbb{R}^{n-1}}} \frac{|f(u+v)-f(u+w)|}{|v|+|w|}.\]
If $\Omega \subset \mathbb{R}^{n}$ we define $\mathrm{reg}_{n-1}f(u,\Omega)$, the \emph{defect of restricted regularity}, by the same formula but with the requirement that the open straight segment $(u+v,u+w)$ lies in the set $\Omega$.
\end{definition}

Intuitively, the regularity defect of $f$ at $u$ describes how much $f$ varies on line segments, with direction in $\mathbb{R}^{n-1}$, which are not too short compared to their distance from $u$.

If $f\colon \mathbb{R}^{n} \to \mathbb{R}^{n-1}$ we denote the derivative of $f$ with respect to the first $n-1$ variables by $f_{n-1}'$. Thus $f_{n-1}'(u)\in L(\mathbb{R}^{n-1},\mathbb{R}^{n-1})$ whenever it exists.

The proof of the following lemma is adapted from the proof of a similar result by \citet*[Lemma 9.5.4, Lemma 13.2.3]{LPT12}. The main difference is that in one direction we have a general doubling measure rather than the Lebesgue measure.

\begin{lemma}\label{pointwiseregularity}
Let $\mu$ be a doubling measure on $\mathbb{R}$ and $f\colon \mathbb{R}^{n} \to \mathbb{R}^{n-1}$ be Lipschitz. Suppose $a \in \mathbb{R}^{n}$ and $R>0$. Then there is a constant $C(\mu)\geq 1$ (independent of $f$, $a$, and $R$) such that
\begin{multline*}
\left( \mathrm{reg}_{n-1}f(a,B(a,R)) \right)^{C(\mu)} \\
\leq C(\mu)(\mathrm{Lip}(f))^{C(\mu)-1}M_{\mathcal{L}^{n-1} \times \mu} (1_{B(a,R)}\|f_{n-1}'\|)(a).
\end{multline*}
\end{lemma}

\begin{proof}
Since the statement is translation invariant and homogeneous with respect to $f$ we can assume $a=0$ and $\mathrm{Lip}(f)=1$. Fix $u, v \in B(0,R)$ with $u-v \in \mathbb{R}^{n-1}$. Suppose $r=|f(u)-f(v)|>0$ and let $S=\max(|u|,|v|)$ and $e=(u-v)/|u-v|$. 

There exists an absolute constant $0<c<1$ 
such that for some $x\in \mathbb{R}^{n}$ and $L>0$, if 
$z \in B(x,2cr)$ then  $|f(z+Le)-f(z)|\geq cr$
and $z+te \in B(0,S)$ for all $0\leq t \leq L$.
Hence, for each $z \in B(x,cr)$,
\[cr \leq \left| \int_{0}^{L} \frac{d}{dt} f(z+te) \dd t \right| \leq  \int_{0}^{L} \|f_{n-1}'(z+te)\| \dd t.\]
Let $V=\{v\in \mathbb{R}^{n-1}: \langle v,e \rangle=0\}$. For each $v\in V\cap B(0,cr)$ we
integrate with respect to $\mu$ to get
\[cr \mu(x_{n}-cr,x_n+cr) \leq \int_{-cr}^{cr}\int_{0}^L \|f_{n-1}'(x + v+te+se_n)\| \dd t \dd \mu(s).\]
Since $\mu$ is doubling, by Lemma \ref{doublingsubballs},
\[\frac{\mu(-S,S)}{\mu(x_{n}-cr,x_n+cr)} \leq c(\mu)(S/cr)^{c(\mu)}.\]
The previous two inequalities imply
\[c r (cr/S)^{c(\mu)} \mu(-S,S)
c(\mu) \leq \int_{-cr}^{cr}\int_{0}^L \|f_{n-1}'(x + v+te+se_n)\| \dd t \dd \mu(s).
\]
Denoting by $\alpha$ the volume of the $n-2$ dimensional unit ball and
integrating over $v\in V\cap B(0,cr)$ with respect to the Lebesgue measure 
leads to
\[\alpha c^{n+c(\mu)-1}r^{n-1} (r/S)^{c(\mu)} \mu(-S,S)
\leq c(\mu)
\int_{B(0,S)} \|f_{n-1}'\| \dd (\mathcal{L}^{n-1} \times \mu),
\]
which we rearrange to
\[(r/S)^{n+c(\mu)-1} \leq \frac{c(\mu)}{\alpha c^{n+c(\mu)-1}S^{n-1} \mu (-S,S)} \int_{B(0,S)} \|f_{n-1}'\| \dd (\mathcal{L}^{n-1} \times \mu), \]
and, observing $(\mathcal{L}^{n-1} \times \mu) (B(0,S))\le 2\alpha S^{n-1} \mu (-S,S)$, 
infer that
\[(r/S)^{n+c(\mu)-1} \leq \frac{2c(\mu)}{c^{n+c(\mu)-1}(\mathcal{L}^{n-1} \times \mu) (B(0,S))} \int_{B(0,S)} \|f_{n-1}'\| \dd (\mathcal{L}^{n-1} \times \mu). \]
With $C(\mu)= n+ 2c(\mu)/c^{n+c(\mu)-1}$ we see from $C(\mu) \geq n+c(\mu)-1$ that
\[\left(\frac{|f(u)-f(v)|}{|u|+|v|}\right)^{C(\mu)} 
\le \left(\frac{|f(u)-f(v)|}{|u|+|v|}\right)^{n+c(\mu)-1}\\
\leq \left(\frac{r}{S}\right)^{n+c(\mu)-1}.\]
Hence
\[
\left(\frac{|f(u)-f(v)|}{|u|+|v|}\right)^{C(\mu)} 
\leq C(\mu)M_{\mathcal{L}^{n-1} \times \mu} (1_{B(0,S)}\|f_{n-1}'\|)(0).\]
By taking a supremum, we conclude
\[ \left( \mathrm{reg}_{n-1}f(0,B(0,R)) \right)^{C(\mu)} 
\leq C(\mu)M_{\mathcal{L}^{n-1} \times \mu} (1_{B(0,R)}\|f_{n-1}'\|)(0)\]
as required.
\end{proof}

The proof of the following proposition is essentially the same as that of a similar result by \citet*[Corollary 13.2.4]{LPT12}.

\begin{proposition}\label{integralregularity}
Let $\Omega \subset \mathbb{R}^{n}$ be a bounded Borel measurable set, $\mu$ be a doubling measure on $\mathbb{R}$, $f\colon \mathbb{R}^{n} \to \mathbb{R}^{n-1}$ be Lipschitz and $s, \lambda >0$. Then
\[\begin{split}
&\mathcal{L}^{n-1} \times \mu \{u \in \mathbb{R}^{n}: B(u,s) \subset \Omega,\, \mathrm{reg}_{n-1} f(u,B(u,s))>\lambda \}\\
& \qquad \leq \frac{C(\mu)(\mathrm{Lip}(f))^{2C(\mu)-2}}{\lambda^{2C(\mu)}} \int_{\Omega} \|f_{n-1}'\|^{2} \dd (\mathcal{L}^{n-1} \times \mu).
\end{split}\]
\end{proposition}

\begin{proof}
For each $u \in \mathbb{R}^{n}$ such that $B(u,s) \subset \Omega$ it follows, using Lemma \ref{pointwiseregularity},
\[ \left(\mathrm{reg}_{n-1}f(u,B(u,s)) \right)^{C(\mu)} \leq C(\mu)(\mathrm{Lip}(f))^{C(\mu)-1}M_{\mathcal{L}^{n-1} \times \mu} (1_{\Omega}\|f_{n-1}'\|)(u).\]
Using the maximal operator and Chebyshev's inequalities give,
\[\begin{split}
\mathcal{L}^{n-1} \times \mu & \{u \in \mathbb{R}^{n}: B(u,s) \subset \Omega,\, \mathrm{reg}_{n-1} f(u,B(u,s))>\lambda \}\\
&\leq \frac{C(\mu)(\mathrm{Lip}(f))^{2C(\mu)-2}}{\lambda^{2C(\mu)}} \int_{\mathbb{R}^{n}} M_{\mathcal{L}^{n-1} \times \mu}(1_{\Omega} \|f_{n-1}'\|)^{2} \dd (\mathcal{L}^{n-1} \times \mu)\\
&\leq \frac{C(\mu)(\mathrm{Lip}(f))^{2C(\mu)-2}}{\lambda^{2C(\mu)}} \int_{\Omega} \|f_{n-1}'\|^{2} \dd (\mathcal{L}^{n-1} \times \mu).
\end{split}\]
\end{proof}

From now on, since increasing the constants weakens the statements, we may assume the constants $C(\mu)$ appearing in statements of results in this section are the same.

\section{Regular Differentiability}

When proving Theorem \ref{theoremdiffinnullset} we will work with functions differentiable only in the direction of certain subspaces. A useful stronger condition is the notion of regular differentiability in a particular direction.

\begin{definition}\label{regularlydifferentiable}
Suppose $f\colon \mathbb{R}^{n} \to \mathbb{R}^{m}$, $T \in L(\mathbb{R}^{p},\mathbb{R}^{n})$ and $x \in \mathbb{R}^{n}$.
\begin{itemize}
	\item We say the function $f$ is \emph{differentiable at $x$ in the direction of $T$} if there is $f'(x;T) \in L(\mathbb{R}^{p},\mathbb{R}^{m})$ such that for every $\varepsilon >0$ there is $\delta > 0$ such that
	\[|f(x+Tv)-f(x)-f'(x;T)(v)|\leq \varepsilon |v|\]
	whenever $v \in \mathbb{R}^{p}$ and $|v|<\delta$.
	\item We say $f$ is \emph{regularly differentiable at $x$ in the direction of $T$} if $f$ is differentiable at $x$ in the direction of $T$ and for every $\varepsilon > 0$ there is $\delta >0$ such that
	\[|f(x+z+Tv)-f(x+z)-f'(x;T)(v)|\leq \varepsilon (|v|+|z|)\]
	whenever $v \in \mathbb{R}^{p}$, $z \in \mathbb{R}^{n}$ and $|v|+|z|<\delta$.
\end{itemize}
\end{definition}

\begin{definition}\label{regularirregular}
Suppose $f\colon \mathbb{R}^{n} \to \mathbb{R}^{m}$. We say that $x$ is a \emph{regular point} of $f$ if for every $v \in \mathbb{R}^{n}$ for which the directional derivative $f'(x;v)$ exists,
\[\lim_{t \to 0} \frac{f(x+tz+tv)-f(x+tz)}{t}=f'(x;v)\]
uniformly for $z \in \mathbb{R}^{n}$ such that $|z|\leq 1$. A point which is not regular is called \emph{irregular}.
\end{definition}

Intuitively regular differentiability of $f$ at $x$ in direction $T$ means the change in $f$ along lines in a direction $Tv$ which pass close to $x$ is well approximated by the directional derivative. Thus the following lemma can be expected.

\begin{lemma}\label{regularpointdiff}
Suppose the map $f\colon \mathbb{R}^{n} \to \mathbb{R}^{m}$ is Lipschitz, $T \in L(\mathbb{R}^{p},\mathbb{R}^{m})$ and $x \in \mathbb{R}^{n}$. Suppose $f$ is differentiable at $x$ in the direction of $T$ and $x$ is a regular point of $f$. Then $f$ is regularly differentiable at $x$ in the direction of $T$. 
\end{lemma}

\begin{proof}
Because $f$ is differentiable at $x$ in the direction of $T$ we have that $f'(x;Tv)$ exists for every $v \in V$ and $f'(x;Tv)=f'(x;T)(v)$. 

Let $\varepsilon>0$. Since $x$ is a regular point of $f$ and the sphere in $\mathbb{R}^{p}$ is compact, there exists $\delta >0$ such that whenever $|v|=\varepsilon$, $0<t<\delta$ and $|z|\leq 1$,
\[|f(x+tz+Ttv)-f(x+tz)-f'(x;T)(tv)|<\varepsilon^{2}t.\]
Set $w=tv$ so $|w|=t\varepsilon$. Then, provided $|w|<\varepsilon \delta$ and $|z|\leq 1$,
\[|f(x+\varepsilon^{-1}|w|z+Tw)-f(x+\varepsilon^{-1}|w|z)-f'(x;T)(w)|<\varepsilon |w|.\]
That is, provided $|w|<\varepsilon \delta$ and $|p|<\varepsilon^{-1}|w|$,
\[|f(x+p+Tw)-f(x+p)-f'(x;T)(w)|<\varepsilon |w|.\]

If $|p|>\varepsilon^{-1}|w|$ then, since $f$ is Lipschitz,
\[|f(x+p+Tw)-f(x+p)-f'(x;T)(w)|\leq 2\mathrm{Lip}(f)\|T\||w|<2\mathrm{Lip}(f)\|T\|\varepsilon |p|.\]

We have shown if $|w|<\varepsilon \delta$ then, for all $p \in \mathbb{R}^{n}$,
\[|f(x+p+Tw)-f(x+p)-f'(x;T)(w)|<\varepsilon |w|+2\mathrm{Lip}(f)\|T\|\varepsilon |p|.\]
Hence $f$ is regularly differentiable at $x$ in the direction $T$.
\end{proof}

Intuitively a set $P$ is porous if each point of $P$ sees nearby, on arbitrarily small scales, relatively large holes in $P$.

\begin{definition}\label{porousdefinition}
A set $P$ in a metric space $M$ is \emph{porous} if for each $x \in P$ there is $\lambda >0$ and $x_{k} \to x$ such that $B(x_{k},\lambda d(x_{k},x)) \cap P = \varnothing$. A set is \emph{$\sigma$-porous} if it is a countable union of porous sets.
\end{definition}

The following fact is well known \citep{Zaj05}; it follows directly from the fact that a Lebesgue density theorem holds for doubling measures.

\begin{proposition}\label{doublingporousnull}
Doubling measures give measure zero to porous sets.
\end{proposition}

The following lemma \citep[Proposition 6.2.6]{LPT12} will be essential when finding regular points of $f$ inside a Lebesgue null set.

\begin{lemma}\label{irregularporous}
Suppose $f\colon \mathbb{R}^{n} \to \mathbb{R}^{m}$ is Lipschitz. Then the set of irregular points of $f$ is $\sigma$-porous.
\end{lemma}

\section{Definition of the Lebesgue Null Set}

As already mentioned, we need our Lebesgue null set to contain plenty of regular points of Lipschitz functions. The following theorem follows from results of \citet{Tuk89} on quasisymmetric mappings.

\begin{theorem}\label{singulardoublingmeasure}
There exists a doubling measure on $\mathbb{R}$ which gives full measure to a Lebesgue null set.
\end{theorem}

Recall that a set in a metric space is of class $G_{\delta}$ if it is a countable intersection of open sets. In the proof of Theorem \ref{theoremdiffinnullset} it will be convenient to work with a $G_{\delta}$ set. The reason for this is the following theorem of Mazurkiewicz \citep[Theorem 8.3]{Dug66}.

\begin{lemma}\label{Mazurkiewicz}
Let $Z$ be a complete metric space. Then $A \subset Z$ is topologically complete if and only if it is a $G_{\delta}$ set in $Z$.
\end{lemma}

It follows from the definitions that every Lebesgue null set is contained in a Lebesgue null set of class $G_{\delta}$.

Let $\mathcal{A}$ be a family of affine maps of the form $T\colon \mathbb{R}^{n} \to \mathbb{R}^{n}$ given by $Tx=Ax+v$, where $A \in L(\mathbb{R}^{n},\mathbb{R}^{n})$ has a matrix representation (with respect to the standard basis of $\mathbb{R}^{n}$) with rational entries and $v \in \mathbb{Q}^{n}$.

We now define the Lebesgue null set we will work with.

\begin{definition}\label{nullset}
Fix a doubling measure $\mu$ on $\mathbb{R}$ which gives full measure to a Lebesgue null set $\widetilde{N} \subset \mathbb{R}$. Let $N$ be any Lebesgue null set in $\mathbb{R}^{n}$ of class $G_{\delta}$ containing
\[\bigcup_{A \in \mathcal{A}} A(\mathbb{R}^{n-1} \times \widetilde{N}).\]
\end{definition}

To prove Theorem \ref{theoremdiffinnullset} we show every Lipschitz function $f\colon \mathbb{R}^{n} \to \mathbb{R}^{n-1}$ has a point of differentiability inside $N$.

\section{Setup of the Variational Principle}

In this section we define a suitable space and perturbations then show the following variational principle \citep[Corollary 7.2.4]{LPT12} is applicable.

\begin{lemma}\label{vpstatement}
Suppose that $f\colon M \to \mathbb{R}$ is lower bounded and lower semicontinuous on a complete metric space $(M,d)$. Suppose further that functions $F_{i}\colon M \times M \to [0,\infty]$, $i\geq 0$, are lower semicontinuous in the second variable with $F_{i}(x,x)=0$ for all $x \in M$ and that $0<r_{i}\leq \infty$ are such that $r_{i} \to 0$ and
\[\inf_{d(x,y)>r_{i}} F_{i}(x,y)>0.\]
If $x_{0} \in M$ and $(\varepsilon_{i})_{i=0}^{\infty}$ is any sequence of positive numbers such that
\[f(x_{0})<\varepsilon_{0}+\inf_{x \in M}f(x) \mbox{ and } \inf_{d(x_{0},y)>r_{0}}F_{0}(x_{0},y)>\varepsilon_{0},\]
then one may find a sequence $(x_{i})_{i=1}^{\infty}$ of points in $M$ converging to some $x_{\infty} \in M$ such that the function
\[h(x)=f(x)+\sum_{i=0}^{\infty} F_{i}(x_{i},x)\]
attains its minimum on $M$ at $x_{\infty}$. Moreover, for each $i \geq 0$,
\[d(x_{i},x_{\infty})\leq r_{i}, \quad F_{i}(x_{i},x_{\infty})\leq \varepsilon_{i}\]
\[h(x_{\infty})\leq \varepsilon_{i}+\inf_{x \in M} \left( f(x)+\sum_{j=0}^{i-1}F_{j}(x_{j},x)\right).\]
\end{lemma}

From now on we intend to make the following assumptions:

\begin{assumptions}\label{assumptions}
Suppose $f\colon \mathbb{R}^{n} \to \mathbb{R}^{n-1}$ is a bounded Lipschitz function and that $0<\eta_{0}<1/2$, $x_{0} \in N$ and $T_{0} \in L(\mathbb{R}^{n-1},\mathbb{R}^{n})$ are such that $\|T_{0}\|=1/2$, $f$ is regularly differentiable at $x_{0}$ in the direction $T_{0}$ and
\[\|\mathrm{Id}-f'(x;T)\|\leq \frac{1}{4}\]
whenever $\|T-T_{0}\| \leq \eta_{0}$ and $f'(x;T)$ exists.
\end{assumptions}

We now show it is sufficient to prove the following proposition.

\begin{proposition}\label{mainprop}
Suppose $f$ satisfies the assumptions above. Then there is a point $x \in N$ such that $|x-x_{0}|\leq \eta_{0}$ and $f$ is differentiable at $x$.
\end{proposition}

\begin{claim}
Theorem \ref{theoremdiffinnullset} follows from Proposition \ref{mainprop}
\end{claim}

\begin{proof}
Suppose $f\colon \mathbb{R}^{n} \to \mathbb{R}^{n-1}$ is Lipschitz. By Lemma \ref{doublingporousnull} and Lemma \ref{irregularporous} the set of regular points of $f$ has full measure with respect to $\mathcal{L}^{n-1} \times \mu$. Hence we may find $t_{0} \in \widetilde{N}$ such that for $\mathcal{L}^{n-1}$ almost every $z \in \mathbb{R}^{n-1}$ the function $f$ is regular at $(z,t_{0})$. Using the classical Rademacher theorem for the Lipschitz function $z \mapsto f(z,t_{0})$ and Lemma \ref{regularpointdiff}, we may find $z_{0} \in \mathbb{R}^{n-1}$ and $T_{0} \in L(\mathbb{R}^{n-1}, \mathbb{R}^{n})$ such that $\mathrm{rank}\, T_{0}=n-1$, $\|T_{0}\|=1/2$ and $f$ is regularly differentiable at $x_{0}=(z_{0},t_{0})$ in the direction $T_{0}$. 

Let $g$ be an extension of $f$ restricted to $B(x_{0},1)$ to a bounded Lipschitz function on $\mathbb{R}^{n}$. Let $R_{0} \in L(\mathbb{R}^{n},\mathbb{R}^{n-1})$ be such that $R_{0}T_{0}=\mathrm{Id}$ on $\mathbb{R}^{n-1}$.

Fix $\eta_{0}>0$ to be chosen small later. Consider the function $h=R_{0}+\eta_{0}g$. Assume that $\|T-T_{0}\|\leq \eta_{0}$. We estimate,
\begin{align*}
\|\mathrm{Id}-h'(x;T)\|&=\|R_{0}(T_{0}-T)-\eta_{0}g'(x;T)\|\\
&\leq \|R_{0}\|\|T-T_{0}\|+\eta_{0}\mathrm{Lip}(g)\|T\|\\
&\leq \|R_{0}\|\|T-T_{0}\|+\eta_{0}\mathrm{Lip}(g)(\|T_{0}\|+\|T-T_{0}\|)\\
&\leq (\|R_{0}\|+2\mathrm{Lip}(g))\eta_{0}\\
&\leq \frac{1}{4}
\end{align*}
provided $\eta_{0}$ is sufficiently small. Hence there exists $x \in N$ with $|x-x_{0}| \leq \eta_{0}$ at which $h$ is differentiable. Consequently $g$ and hence $f$ are differentiable at $x \in N$, as required.
\end{proof}

From now on we make the assumptions given above and focus on proving Proposition \ref{mainprop}. We now establish basic consequences of our assumptions. First we state a mean value estimate \citep[Proposition 2.4.1]{LPT12}.

\begin{lemma}\label{meanvalueestimate}
Let $\Lambda$ be a real valued locally Lipschitz function on an open subset $G$ of a separable Banach space $X$, and let $a, b \in G$ be such that the straight segment from $a$ to $b$ is contained in $G$. Then for every $\varepsilon>0$ there is a point $z \in G$ at which $\Lambda$ is G\^{a}teaux differentiable and
\[\Lambda'(z)(b-a)>\Lambda(b)-\Lambda(a)-\varepsilon.\]
\end{lemma}

\begin{proposition}\label{consequencesofassumptions}
The following facts hold:
\begin{enumerate}[(1)]
	\item The inequality $|f(x+Tu)-f(x)|\leq 5|u|/4$ holds for every $x \in \mathbb{R}^{n}$, linear map $T$ such that $\|T-T_{0}\|\leq \eta_{0}$, and $u \in \mathbb{R}^{n-1}$.
	\item If $\|T-T_{0}\|\leq \eta_{0}$ and $f'(x;T)$ exists then $T$ is a linear isomorphism of $\mathbb{R}^{n-1}$ onto its image and $f'(x;T)$ is a linear isomorphism of $\mathbb{R}^{n-1}$ onto itself.
	\item Necessarily $\mathrm{Lip}(f) \geq 1$.
\end{enumerate}
\end{proposition}

\begin{proof}
Temporarily denote $v=(f(x+Tu)-f(x))/|f(x+Tu)-f(x)| \in \mathbb{R}^{n-1}$. The map $z \mapsto v\cdot f(x+Tz)$ is a real valued Lipschitz map on $\mathbb{R}^{n-1}$. Hence, by Lemma \ref{meanvalueestimate} with $a=0$ and $b=u$, for every $\varepsilon>0$ there is a point $z \in \mathbb{R}^{n-1}$ at which this map is differentiable and
\[v\cdot (f'(x+Tz;T)u) > v\cdot (f(x+Tu)-f(x))-\varepsilon.\]
Hence
\[|f(x+Tu)-f(x)|\leq \sup \{\|f'(y;T)\|: y \in \mathbb{R}^{n}, f'(y;T) \mbox{ exists}\}|u| \leq \frac{5}{4}|u|\]
which proves (1).

Suppose $\|T-T_{0}\| \leq \eta_{0}$ and $f'(x;T)$ exists. Then by our assumptions
\[\|\mathrm{Id}-f'(x;T)\|\leq \frac{1}{4}.\]
It follows immediately from this that $f'(x;T)$ is a linear injection from $\mathbb{R}^{n-1}$ to itself and consequently is a linear isomorphism. Using the definition of $f'(x;T)$ it follows that $T$	must have $n-1$ dimensional image, hence is a linear isomorphism of $\mathbb{R}^{n-1}$ onto its image. Hence (2) is true.

If $u$ is any unit vector in $\mathbb{R}^{n-1}$ then $2\|T_{0} u\| \leq 1$. Hence
\[\mathrm{Lip}(f) \geq 2\|f'(x_{0};T_{0}u)\| \geq 2(|u|-|u-f'(x_{0};T_{0}u)|) \geq 1.\]
This proves (3).
\end{proof}

Recall that if $H$ is a Hilbert space then the \emph{Hilbert-Schmidt norm} $\|\cdot\|_{H}$ and corresponding inner product $\langle \cdot,\cdot \rangle_{H}$ are defined on the space $L(H,\mathbb{R}^{n-1})$ of bounded linear operators from $H$ to $\mathbb{R}^{n-1}$ and given by,
\[\|T\|_{H}^{2}=\sum_{i=1}^{\infty} |Tu_{i}|^{2} \mbox{ and } \langle T,S\rangle_{H} = \sum_{i=1}^{\infty} \langle Tu_{i},Su_{i} \rangle\]
for any orthonormal basis $(u_{i})$ of $H$; the value of the Hilbert-Schmidt norm and inner product are independent of the orthonormal basis used \citep{LPT12}.

We define the space
\[\begin{split}
D=\{(x,T) \in N \times & L(\mathbb{R}^{n-1}, \mathbb{R}^{n}):\\
&f \mbox{ is regularly differentiable at } x \mbox{ in direction } T\}.
\end{split}\]
We give $\mathbb{R}^{n} \times \mathbb{R}^{n-1}$ the norm $\|(z,w)\|=|z|+|w|$ for $z \in \mathbb{R}^{n}$ and $w\in \mathbb{R}^{n-1}$. Define a pseudonorm $\|\cdot\|_{r}$ on the space $\mathrm{Lip}(\mathbb{R}^{n}\times \mathbb{R}^{n-1},\mathbb{R}^{n-1})$ of $\mathbb{R}^{n-1}$ valued Lipschitz functions on $\mathbb{R}^{n} \times \mathbb{R}^{n-1}$ by
\[\|h\|_{r}= \sup_{\substack{z \in \mathbb{R}^{n}, u \in \mathbb{R}^{n-1}\\ |z|+|u|>0}} \frac{|h(z,u)-h(z,0)|}{|z|+|u|}.\]
It is easy to see this pseudonorm is lower semicontinuous in the topology of pointwise convergence and $\|h\|_{r} \leq \mathrm{Lip}(h)$.

For $x \in \mathbb{R}^{n}$ and $T \in L(\mathbb{R}^{n-1}, \mathbb{R}^{n})$ we define $f_{x,T}\colon \mathbb{R}^{n} \times \mathbb{R}^{n-1} \to \mathbb{R}^{n-1}$ by
\[f_{x,T}(z,u)=f(x+z+Tu).\]
The transformation $(x,T) \mapsto f_{x,T}$ maps $\mathbb{R}^{n} \times L(\mathbb{R}^{n-1},\mathbb{R}^{n})$ to the space $\mathrm{Lip}(\mathbb{R}^{n}\times \mathbb{R}^{n-1}, \mathbb{R}^{n-1})$. We have the inequalities
\[\|f_{x,T}-f_{y,S}\|_{r} \leq \mathrm{Lip}(f) (\|T\|+\|S\|)\]
and
\[\|f_{x,T}-f_{x,S}\|_{r} \leq \mathrm{Lip}(f) \|T-S\|.\]

Recall $N$ is a $G_{\delta}$ set in $\mathbb{R}^{n}$. Hence, by Lemma \ref{Mazurkiewicz}, there is a metric $d_{N}$ on $N$ such that $(N,d_{N})$ is complete and $d_{N}$ is topologically equivalent to the Euclidean metric on $N$. Further, by replacing $d_{N}$ by the metric $(x,y) \mapsto d_{N}(x,y)+|x-y|$ if necessary, we may assume $|x-y|\leq d_{N}(x,y)$ for all $x, y \in N$.

The space $D$ has a standard topology as a subset of $\mathbb{R}^{n} \times L(\mathbb{R}^{n-1},\mathbb{R}^{n})$. We also define a metric, giving a different topology, on $D$ by
\[d((x,T),(y,S))=d_{N}(x,y)+\|T-S\|+\|f'(x;T)-f'(y;S)\|_{H}+\|f_{x,T}-f_{y,S}\|_{r}.\]

Note the metric $d$ is similar to the one in \citep{LPT12} with the norm distance replaced by $d_{N}$ and the regularity component $\| \cdot \|_{r}$ slightly simplified.

We now establish several properties of $(D,d)$.

\begin{lemma}\label{Dlscsep}
The following facts hold:
\begin{enumerate}[(1)]
	\item The function $(x,T) \mapsto f'(x;T)$ is $d$-continuous on $D$.
	\item The function $((x,T),(y,S)) \mapsto \|f_{x,T}-f_{y,S}\|_{r}$ is lower semicontinuous in the standard topology of $D \times D$.
	\item The function $((x,T),(y,S)) \mapsto \|f_{x,T}-f_{y,T}\|_{r}$ is lower semicontinuous in the standard topology of $D \times D$, and so also in the topology of the product $(D,d) \times (D,d)$.
	\item The space $(D,d)$ is separable.
\end{enumerate}
\end{lemma}

\begin{proof}
(1) follows from the definition of $d$.

Since $f$ is continuous, for each fixed $z \in \mathbb{R}^{n}$ and $u \in \mathbb{R}^{n-1}$ with $|z|+|u|>0$ the function
\[((x,T),(y,S)) \mapsto \frac{|f_{x,T}(z,u)-f_{y,S}(z,u)-f_{x,T}(z,0)+f_{y,S}(z,0)|}{|z|+|u|}\]
is continuous in the standard topology of $D \times D$. The function
\[((x,T),(y,S)) \mapsto \|f_{x,T}-f_{y,S}\|_{r}\]
is a supremum of a family of continuous functions, hence lower semicontinuous. This proves (2).

The function
\[((x,T),(y,S)) \mapsto \|f_{x,T}-f_{y,T}\|_{r}\]
is lower semicontinuous in the standard topology of $D \times D$ by a similar argument to that used in (2). Since the topology of $(D,d) \times (D,d)$ has more open sets than the standard topology of $D \times D$, (3) follows easily.

Let $H_{0}$ be the space of continuous functions $\varphi\colon \mathbb{R}^{n} \times \mathbb{R}^{n-1} \to \mathbb{R}^{n-1}$. We equip $H_{0}$ with topology generated by a countable family of pseudonorms
\[ \| \varphi \|_{k} = \sup_{\substack{z \in \mathbb{R}^{n}, u \in \mathbb{R}^{n-1}\\ |z|+|u|<k}} |\varphi(z,u)|.\]
The space $H_{0}$ can be metrized by the metric
\[\rho_{0}(\varphi, \psi)=\sum_{k=1}^{\infty} 2^{-k} \min\{1,\, \|\varphi - \psi\|_{k}\}\]
and a countable dense subset of $H_{0}$ is given by polynomials with rational coefficients.

Hence the space
\[H=\mathbb{R}^{n} \times L(\mathbb{R}^{n-1},\mathbb{R}^{n}) \times L(\mathbb{R}^{n-1}, \mathbb{R}^{n-1}) \times H_{0}\]
is also metrizable and separable. We show $(D,d)$ is homeomorphic to a topological subspace of $H$, which implies, since subspaces of separable metric spaces are separable, $(D,d)$ is separable.

For $(x,T) \in D$ define $\psi_{x,T}\colon \mathbb{R}^{n} \times \mathbb{R}^{n-1} \to \mathbb{R}^{n-1}$ by
\[\psi_{x,T}(z,u)=\frac{f_{x,T}(z,u)-f_{x,T}(z,0)-f'(x;Tu)}{|z|+|u|},\, \psi_{x,T}(0,0)=0.\]
Since $f$ is regularly differentiable at $x$ in direction $T$ it follows that $\psi_{x,T}$ is continuous at $(0,0)$. Hence $\psi_{x,T} \in H_{0}$ so the map $\eta\colon D \to H$ given by
\[\eta(x,T)=(x,T,f'(x;T),\psi_{x,T})\]
is an injection from $D$ into $H$. We claim $\eta$ is a homeomorphism onto its image. From the definition of $d$ it is clear the first three components of $\eta$ are $d$ continuous. Continuity of the last component follows from the inequality
\[\|\psi_{x,T}-\psi_{y,S}\|_{k}\leq \|f_{x,T}-f_{y,S}\|_{r}+\|f'(x;T)-f'(y;S)\|.\]
Continuity of $\eta^{-1}$ follows from the estimate
\begin{align*}
\|f_{x,T}-f_{y,S}\|_{r} &\leq \sup_{\substack{z \in \mathbb{R}^{n}, u \in \mathbb{R}^{n-1}\\ |z|+|u|\geq k}} \frac{|f_{x,T}(z,u)-f_{x,T}(z,0)-f_{y,S}(z,u)+f_{y,S}(z,0)|}{|z|+|u|}\\
&\qquad + \|\psi_{x,T}-\psi_{y,S}\|_{k}+\|f'(x;T)-f(y;S)\|\\
&\leq \frac{4\|f\|_{\infty}}{k} + \|\psi_{x,T}-\psi_{y,S}\|_{k}+\|f'(x;T)-f'(y;S)\|
\end{align*}
for every $k \geq 1$. This completes the proof of (4).
\end{proof}

\begin{lemma}\label{Dcomplete}
The pair $(D,d)$ is a complete metric space.
\end{lemma}

\begin{proof}
Suppose $\varepsilon_{i} \downarrow 0$ and $(x_{i},T_{i}) \in D$ are such that
\[d((x_{j},T_{j}),(x_{i},T_{i}))< \varepsilon_{i} \mbox{ whenever }j \geq i.\]
Since $(N,d_{N})$ is complete it follows $x_{i}$ converges to some $x \in N$. Clearly, from the definition of $d$, $T_{i}$ converges to some $T \in L(\mathbb{R}^{n-1}, \mathbb{R}^{n})$ and the maps $L_{i}=f'(x_{i};T_{i})$ converge to some $L \in L(\mathbb{R}^{n-1}, \mathbb{R}^{n-1})$. Using the definition of $d$, $\|L-L_{i}\| \leq \varepsilon_{i}$. Also, using Lemma \ref{Dlscsep} (2) and the definition of $d$, $\|f_{x,T}-f_{x_{i},T_{i}}\|_{r} \leq \varepsilon_{i}$.

To complete the proof we must show $(x,T) \in D$ and $(x_{i},T_{i})$ converges to $(x,T)$ with respect to the metric $d$. To show $(x,T) \in D$ we show $f$ is regularly differentiable at $x$ in direction $T$ with $f'(x;T)=L$. Suppose $\varepsilon>0$ and find $i$ such that $\varepsilon_{i}<\varepsilon/3$. Since $f$ is regularly differentiable at $x_{i}$ in direction $T_{i}$, we can find $\delta>0$ such that
\[|f(x_{i}+z+T_{i}v)-f(x_{i}+z)-L_{i}v|\leq \frac{\varepsilon}{3}(|z|+|v|)\]
whenever $|z|+|v|<\delta$. Hence, using the definition of $\|\cdot\|_{r}$, for $|z|+|v|<\delta$,
\begin{align*}
&|f(x+z+Tv)-f(x+z)-Lv|\\
& \qquad \leq |f(x_{i}+z+T_{i}v)-f(x_{i}+z)-L_{i}v|+|L_{i}v-Lv|\\
& \qquad \qquad + \|f_{x,T}-f_{x_{i},T_{i}}\|_{r}(|z|+|v|)\\
& \qquad \leq \frac{\varepsilon}{3} (|z|+|v|) + \varepsilon_{i}|v| + \varepsilon_{i}(|z|+|v|)\\
& \qquad \leq \varepsilon (|z|+|v|).
\end{align*}
Hence $f$ is regularly differentiable at $x$ in direction $T$.

Convergence of $(x_{i}, T_{i})$ to $(x,T)$ in the metric $d$ follows from the estimate $\|f_{x,T}-f_{x_{i},T_{i}}\|_{r} \leq \varepsilon_{i}$.
\end{proof}

Since $(D,d)$ is a complete separable metric space and the identity map from $(D,d)$ to $\mathbb{R}^{n} \times L(\mathbb{R}^{n-1},\mathbb{R}^{n})$ is continuous, it follows a subset of $D$ is Borel in $(D,d)$ if and only if it is Borel in $\mathbb{R}^{n} \times L(\mathbb{R}^{n-1},\mathbb{R}^{n})$ \citep{Kec95}. In what follows it will be obvious functions we integrate are Borel in $(D,d)$ so it follows by this fact that they are also Borel in the standard topology on $\mathbb{R}^{n} \times L(\mathbb{R}^{n-1},\mathbb{R}^{n})$.

Fix a Lipschitz function $\Theta\colon  L(\mathbb{R}^{n-1},\mathbb{R}^{n}) \to [0,1]$ which is differentiable everywhere and
\[\inf_{\|S\|>s} \Theta(S)>\Theta(0)=0 \mbox{ for every } s>0.\]

In the context of \citep{LPT12}, existence of such a bump function, defined on a Banach space (and required to satisfy certain smoothness assumptions) was needed as a hypothesis on the Banach space. Existence of such a function in our case is clear - for example, identify $L(\mathbb{R}^{n-1},\mathbb{R}^{n})$ with $\mathbb{R}^{n(n-1)}$ and let
\[\Theta(S)=\|S\|_{E}^{2}/(1+\|S\|_{E}^{2})\]
where $\|\cdot\|_{E}$ is the Euclidean norm.

We work in the subspace $(D_{0},d)$ of $(D,d)$ where
\[D_{0}=\{(x,T) \in D: d_{N}(x,x_{0})\leq \eta_{0},\, \|T-T_{0}\| \leq \eta_{0}\}\]
with $\eta_{0}$, $x_{0}$ and $T_{0}$ defined in Assumptions \ref{assumptions}.

Temporarily fix parameters $0<\lambda_{i}, \beta_{i}, \gamma_{i}, \sigma_{i}, s_{i}<\infty$ to be chosen later. We will apply the variational principle of Lemma \ref{vpstatement} on the complete metric space $(D_{0},d)$ to the function $h_{0}\colon D_{0} \to \mathbb{R}$ given by
\[h_{0}(x,T)=-\|f'(x;T)\|^{2}_{H},\]
with the perturbation functions $F_{i}\colon D_{0} \times D_{0} \to [0,\infty),\, i \geq 0$, defined by
\begin{align*}
F_{i}((x,T),(y,S))=& \lambda_{i}d_{N}(x,y)+\beta_{i}\Theta(S-T)\\
&\quad +\gamma_{i}\|f'(y;S)-f'(x;T)\|_{H}^{2}+\sigma_{i}\Delta_{i}((x,T),(y,S)),
\end{align*}
where
\[\Delta_{i}((x,T),(y,S))=\max\{0,\,\min\{1,\, \|f_{y,T}-f_{x,T}\|_{r}-s_{i}\}\}.\]

We now show the variational principle can be applied and prove estimates that will be useful later.

\begin{lemma}\label{vpsatisfied}
Suppose that for all $i \geq 0$,
\[0<\lambda_{i}, \beta_{i}, \gamma_{i}, \sigma_{i}, \varepsilon_{i}<\infty,\, 0\leq s_{i}< \infty,\, s_{i} \downarrow 0.\]
Then $F_{i}$ are non-negative lower semicontinuous functions on $(D_{0},d) \times (D_{0},d)$ satisfying $F_{i}((x,T),(x,T))=0$ and there are $r_{i} \downarrow 0$ such that
\[\inf \{F_{i}((x,T),(y,S)):d((x,T),(y,S))\geq r_{i}\}>0.\]
If, moreover,
\[\|f'(x_{0};T_{0})\|_{H}^{2} > \sup_{(x,T) \in D_{0}} \|f'(x;T)\|_{H}^{2}-\varepsilon_{0}.\]
then the function $h_{0}$ and the perturbation scheme $(F_{i})$ satisfy the assumptions of the variational principle of Lemma \ref{vpstatement} on the metric space $(D_{0},d)$.
\end{lemma}

\begin{proof}
Clearly $F_{i}\geq 0$ and $F_{i}((x,T),(x,T))=0$. Lower semicontinuity of $\Delta_{i}$ follows from Lemma \ref{Dlscsep} (3). The other terms of $F_{i}$ are clearly continuous with respect to $d$.

Let $r_{0}=\infty$. For $i \geq 1$ let $t_{i}=1/2^{i+1}$; we show that $r_{i}=s_{i}+(4+\mathrm{Lip}(f))t_{i}$ satisfy
\[\inf_{d((x,T),(y,S))\geq r_{i}} F_{i}((x,T),(y,S)) \geq \min \{\lambda_{i}t_{i},\, \gamma_{i}t_{i}^{2},\, \sigma_{i}t_{i},\, \inf_{\|L\|\geq t_{i}} \beta_{i}\Theta(L)\}.\]
This is obvious, from the definition of $F_{i}$, if $d_{N}(x,y)\geq t_{i}$ or $\|T-S\| \geq t_{i}$ or $\|f'(x;T)-f'(y;S)\|_{H} \geq t_{i}$. Suppose $d((x,T),(y,S))\geq r_{i}$ and none of the previous inequalities hold. Then
\[\|f_{x,T}-f_{y,S}\|_{r} \geq d((x,T),(y,S))-3t_{i} \geq s_{i}+(1+\mathrm{Lip}(f))t_{i}.\]
Hence
\begin{align*}
\|f_{x,T}-f_{y,T}\|_{r} &\geq \|f_{x,T}-f_{y,S}\|_{r}-\|f_{y,S}-f_{y,T}\|_{r}\\
&\geq s_{i}+(1+\mathrm{Lip}(f))t_{i}-\mathrm{Lip}(f)\|S-T\| \\
&\geq s_{i}+t_{i}.
\end{align*}
Hence $\|f_{x,T}-f_{y,T}\|_{r} \geq s_{i}+t_{i}$ and $\sigma_{i}\Delta_{i}((x,T),(y,S)) \geq \sigma_{i} t_{i}$.

Since $r_{0}=\infty$, the final condition assumed is the only remaining requirement of the variational principle.
\end{proof}

\section{Application of the Variational Principle}

We now apply the variational principle, derive some resulting estimates and finally make exact choices of parameters.

Assume the parameters $\lambda_{i}, \beta_{i}, \gamma_{i}, \sigma_{i}, s_{i}, \varepsilon_{i}, \varepsilon_{0}$ satisfy the assumptions of Lemma \ref{vpsatisfied}. Then the variational principle shows that $(x_{0},T_{0})$ is the starting term of a sequence $(x_{j}, T_{j}) \in D_{0}$ which $d$-converges to some $(x_{\infty},T_{\infty}) \in D_{0}$ and has the property that, denoting $\varepsilon_{\infty}=0$ and
\[h_{i}(x,T)=-\|f'(x;T)\|_{H}^{2}+\sum_{j=0}^{i-1} F_{j}((x_{j},T_{j}),(x,T)),\]
we have,
\[h_{\infty}(x_{\infty},T_{\infty}) \leq \varepsilon_{i}+\inf_{(x,T)\in D_{0}} h_{i}(x,T)\]
for $0 \leq i \leq \infty$. Note that for $i=\infty$ this says $(x_{\infty},T_{\infty})$ is a minimum of $h_{\infty}$ in $D_{0}$. Since $(x_{i}, T_{i}) \in D$, the derivatives $L_{i}=f'(x_{i};T_{i})$ exist for all $0\leq i \leq \infty$.

We guarantee $h_{\infty}$ is finite by requiring
\[\sum_{i=0}^{\infty} (\lambda_{i}+\beta_{i}+\gamma_{i}+\sigma_{i}) < \infty.\]

\begin{lemma}\label{dtlestimatesgeneral}
For every $i \geq 0$,
 \[d_{N}(x_{\infty},x_{i})\leq \frac{\varepsilon_{i}}{\lambda_{i}}, \quad \Theta (T_{\infty}-T_{i})\leq \frac{\varepsilon_{i}}{\beta_{i}}, \quad \|L_{\infty}-L_{i}\|_{H}\leq \sqrt{\frac{\varepsilon_{i}}{\gamma_{i}}}.\]
\end{lemma}

\begin{proof}
It follows from the definition of $h_{i}$ that
\[\begin{split}
h_{i}(x_{\infty},T_{\infty})+\lambda_{i}d_{N}(x_{\infty},x_{i})+\beta_{i}\Theta(T_{\infty}-T_{i})+\gamma_{i}\|L_{\infty}-L_{i}\|_{H}^{2}\\
\leq h_{\infty}(x_{\infty},T_{\infty}) \leq \varepsilon_{i}+h_{i}(x_{\infty},T_{\infty}).
\end{split}\]
Since all terms are non-negative this implies all three estimates.
\end{proof}

Our strategy will be to suppose $f$ is not differentiable at $x_{\infty}$ and show this would imply $x_{\infty}$ is not a minimum of the function $h_{\infty}$. To do this we analyze the difference,
\begin{equation}\label{changeinh}
\begin{split}h_{\infty}(x_{\infty},T_{\infty})-h_{\infty}(x,T)=&(\|f'(x;T)\|_{H}^{2}-\|f'(x_{\infty};T_{\infty})\|_{H}^{2})\\
&\quad -\Phi(x)-\Psi(T)-\Upsilon(f'(x;T))-\Delta(x)\end{split}
\end{equation}
where, for $x \in N$, $T \in L(\mathbb{R}^{n-1},\mathbb{R}^{n})$ and $L \in L(\mathbb{R}^{n-1},\mathbb{R}^{n-1})$,
\[\Phi(x)=\sum_{i=0}^{\infty} \lambda_{i}(d_{N}(x,x_{i})-d_{N}(x_{\infty},x_{i}))\]
\[\Psi(T)=\sum_{i=0}^{\infty} \beta_{i}(\Theta(T-T_{i})-\Theta(T_{\infty}-T_{i}))\]
\[\Upsilon(L)=\sum_{i=0}^{\infty} \gamma_{i}(\|L-L_{i}\|_{H}^{2}-\|L_{\infty}-L_{i}\|_{H}^{2})\]
\[\Delta(x)=\sum_{i=0}^{\infty} \sigma_{i}(\Delta_{i}((x_{i},T_{i}),(x,T))-\Delta_{i}((x_{i},T_{i}),(x_{\infty},T_{\infty}))).\]

These functions are well defined and one sees from the definition of $\Delta_{i}$ that $\Delta$ does not depend on $T$. Also note the positive and finite function
\[\Theta_{\infty}(T)=\sum_{i=0}^{\infty} \beta_{i}\Theta(T-T_{i})\]
is differentiable (by standard facts about uniformly converging series of functions).

We now give definitions of the parameters. For $i \geq 1$,
\[\lambda_{i}=2^{-i}\lambda_{0},\quad \beta_{i}=2^{-i}\beta_{0},\quad \gamma_{i}=2^{-i}\gamma_{0},\quad s_{i}=2^{-i}s_{0}.\]

Choose $\varepsilon_{0}>0$ such that
\[\|f'(x_{0},T_{0})\|_{H}^{2} > \sup_{(x,T) \in D_{0}} \|f'(x;T)\|_{H}^{2}-\varepsilon_{0}.\]

Let $\lambda_{0}=2\varepsilon_{0}/\eta_{0}$, $\gamma_{0}=1/8$ and $s_{0}=4$. Then find $\beta_{0}>0$ large enough that
\[\Theta(S)>\frac{\varepsilon_{0}}{\beta_{0}} \mbox{ whenever } \|S\|>\min \left\{\frac{\eta_{0}}{2},\,\frac{s_{0}}{8\mathrm{Lip}(f)}\right\}.\]

For $i\geq 1$ we choose $\varepsilon_{i}>0$ such that
\[\frac{\varepsilon_{i}}{\lambda_{i}}\leq \frac{\eta_{0}}{2},\quad \frac{\varepsilon_{i}}{\gamma_{i}}\leq \frac{s_{i}^{2}}{64}\]
and
\[\Theta(S)>\frac{\varepsilon_{i}}{\beta_{i}} \mbox{ whenever } \|S\|>\min \left\{ \frac{\eta_{0}}{2},\, \frac{s_{i}}{8\mathrm{Lip}(f)}\right\}.\]

We now deduce estimates about the speed of convergence of $x_{i}$, $T_{i}$ and $L_{i}$.

\begin{lemma}\label{dtlestimatesapplied}
For all $0\leq i < \infty$ we have the following estimates of the speed of convergence:
\[d_{N}(x_{i},x_{\infty}) \leq \frac{\eta_{0}}{2},\quad \|T_{i}-T_{\infty}\|\leq \min \left\{ \frac{\eta_{0}}{2},\, \frac{s_{i}}{8\mathrm{Lip}(f)}\right\},\]
and
\[\|L_{i}-L_{\infty}\|\leq \frac{s_{i}}{8}.\]
Also, for $0\leq i \leq \infty$,
\[\|T_{i}\| \leq 1,\quad \|L_{i}\|\leq \frac{5}{4}.\]
\end{lemma}

\begin{proof}
The first two inequalities, for $0\leq i < \infty$, and the third inequality, for $0< i <\infty$, follow from the definitions and Lemma \ref{dtlestimatesgeneral}. Using the definition of $f'(x;T)$ and the second inequality for $i=0$ we obtain
\[\|L_{0}-L_{\infty}\| \leq \mathrm{Lip}(f)\|T_{0}-T_{\infty}\| \leq \frac{s_{0}}{8}\]
which is the third inequality for $i=0$. Again, using the second inequality we have, for any $0 \leq i \leq \infty$,
\[\|T_{i}\|\leq \|T_{0}\|+\|T_{0}-T_{\infty}\|+\|T_{i}-T_{\infty}\|\leq \frac{1}{2}+\eta_{0}\leq 1.\]
Since $(x_{i},T_{i}) \in D_{0}$ for all $0 \leq i \leq \infty$, the final estimate follows from the assumptions on $f$ (assumptions \ref{assumptions}).
\end{proof}

Finally for $i \geq 0$ we choose $\sigma_{i}>0$ such that

\[\sigma_{i}\frac{C(\mu)(5\mathrm{Lip}(f))^{2C(\mu)-2}}{(s_{i}/8)^{2C(\mu)}}\leq 2^{-i-3},\]
\[\sigma_{i}\leq 2^{-i-4},\]
and
\[\sigma_{i}\leq \frac{2^{-i-5}}{i+1}\frac{s_{i}}{48\mathrm{Lip}(f)}.\]

\section{Non-Differentiability Contradicts Minimality}

In this section we suppose $f$ is not differentiable at $x_{\infty}$ and deduce consequences. We then define a region in which we plan to find a point $x$, with corresponding direction $T$, such that $(x,T) \in D_{0}$ and $h_{\infty}(x,T)<h_{\infty}(x_{\infty},T_{\infty})$. This contradiction would then show $f$ must be differentiable at $x_{\infty}$. Most of the arguments are similar to those in \citep{LPT12} except we must choose the region mentioned so that it contains many points of $N$ which are regular points of the function $f$.

Let $L_{\Theta}$ be the derivative of $\Theta_{\infty}$ at $T_{\infty}$. We recall
\begin{align*}h_{\infty}(x,T)&=-\|f'(x;T)\|_{H}^{2}+ \sum_{i=0}^{\infty} \lambda_{i}d_{N}(x,x_{i})+\sum_{i=0}^{\infty}\beta_{i}\Theta(T-T_{i})\\
&\qquad +\sum_{i=0}^{\infty} \gamma_{i}\|f'(x;T)-L_{i}\|_{H}^{2}\\
&\qquad \qquad +\sum_{i=0}^{\infty} \sigma_{i} \max\{ 0,\, \min \{1,\, \|f_{x,T_{i}}-f_{x_{i},T_{i}}\|_{r}\}-s_{i}\}.
\end{align*}

Suppose for the moment $f$ is differentiable at $x_{\infty}$. Since 
by Lemma~\ref{dtlestimatesapplied} $(x_{\infty}, T)\in D_0$ for $\|T-T_\infty\|<\frac12\eta_0$ and
since $h_{\infty}$ attains its minimum on $D_0$ at $(x_{\infty}, T_{\infty})$, 
this implies that the function $T \mapsto h_{\infty}(x_{\infty},T)$ is differentiable with derivative zero at $T_{\infty}$.
We differentiate to find equations for $f'(x_{\infty};S)$. Observe
\[\|A_{0}+A\|_{H}^{2}-\|A_{0}\|_{H}^{2}-2\langle A_{0},A\rangle_{H}=\|A\|_{H}^{2}\]
for linear maps $A, \, A_{0} \in L(\mathbb{R}^{n-1}, \mathbb{R}^{n-1})$. This implies the derivative of the map
\[A \mapsto \|A\|_{H}^{2}\]
at $A_{0}$ is the map
\[A \mapsto 2\langle A_{0}, A\rangle_{H}.\]
Hence, using the chain rule and the fact $T\mapsto f'(x;T)$ is linear, we deduce
\begin{equation}\label{derivativeequation}
-2\langle L_{\infty},f'(x_{\infty};S)\rangle_{H}+L_{\Theta}S+2\sum_{i=0}^{\infty}\gamma_{i} \langle L_{\infty}-L_{i}, f'(x_{\infty};S)\rangle_{H}=0
\end{equation}
for all $S \in L(\mathbb{R}^{n-1}, \mathbb{R}^{n})$. We now solve this system of equations for an unknown operator $L \in L(\mathbb{R}^{n}, \mathbb{R}^{n-1})$ in place of $f'(x_{\infty})$. Denote
\[R_{\infty}=2L_{\infty}+2\sum_{i=0}^{\infty} \gamma_{i}(L_{i}-L_{\infty}) \in L(\mathbb{R}^{n-1}, \mathbb{R}^{n-1}).\]
Then we can rewrite the system of equations as
\begin{equation}\label{unknownoperator}
\langle R_{\infty},LS\rangle_{H}=L_{\Theta}S \mbox{ for all }S \in L(\mathbb{R}^{n-1}, \mathbb{R}^{n}).
\end{equation}

Since every $S \in L(\mathbb{R}^{n-1}, \mathbb{R}^{n})$ is a sum of rank one operators, we may consider this linear equation only for rank one operators $S \in L(\mathbb{R}^{n-1}, \mathbb{R}^{n})$. We first deduce $R_{\infty}$ is invertible.

\begin{lemma}
The operator $R_{\infty}$ is invertible and $\|R_{\infty}\|^{-1} \leq 1$.
\end{lemma}

\begin{proof}
By Lemma \ref{dtlestimatesapplied} we know $\|L_{i}-L_{\infty}\| \leq 1/2$ for all $i\geq 0$. Since $\sum_{i=0}^{\infty} \gamma_{i}=1/4$, and $\|\mathrm{Id}-L_{\infty}\| \leq 1/4$ by assumptions \ref{assumptions}, we obtain,
\begin{align*}
\|\mathrm{Id}-(1/2)R_{\infty}\|&=\|\mathrm{Id}-L_{\infty}-\sum_{i=0}^{\infty} \gamma_{i}(L_{i}-L_{\infty})\|\\
&\leq \|\mathrm{Id}-L_{\infty}\|+\sum_{i=0}^{\infty} \gamma_{i} \|L_{i}-L_{\infty}\|\\
&\leq 1/2.
\end{align*}
This implies the expression
\[\sum_{i=0}^{\infty}(\mathrm{Id}-(1/2)R_{\infty})^{i}\]
is a well defined linear map with norm at most $2$. An easy computation then shows it is the inverse of $(1/2)R_{\infty}$. Hence the map $R_{\infty}$ is invertible with $\|R_{\infty}^{-1}\|\leq 1$.
\end{proof}

A rank one operator $S \in L(\mathbb{R}^{n-1}, \mathbb{R}^{n})$ can be written as $S=x\otimes e$ where $x \in \mathbb{R}^{n}$ and $e \in \mathbb{R}^{n-1}$. It acts by $(x \otimes e)(u)=\langle e,u \rangle x$. We note $L(x\otimes e)=Lx\otimes e$ for a linear map $L$ defined on $\mathbb{R}^{n}$. Indeed, for any $u \in \mathbb{R}^{n-1}$,
\[L(x\otimes e)u=L(\langle e, u \rangle x)=\langle e, u \rangle Lx = (Lx \otimes e)(u).\]
Also $\langle R,w \otimes e\rangle_{H}=\langle Re,w\rangle$ for any $w \in \mathbb{R}^{n}$ and $R \in L(\mathbb{R}^{n-1}, \mathbb{R}^{n})$. To see this let $e_{i}$ be an orthonormal basis of $\mathbb{R}^{n-1}$. Then,
\[\langle R,w \otimes e\rangle_{H} = \sum_{i=1}^{n-1} \langle R(e_{i}), \langle e_{i}, e \rangle w \rangle = \langle R(\sum_{i=1}^{n-1} \langle e_{i}, e \rangle e_{i}), w \rangle = \langle Re, w\rangle.\]

Using these facts, the system of equations \eqref{unknownoperator} for the unknown $Lx$ can be written
\[\langle R_{\infty}e,Lx\rangle = L_{\Theta}(x \otimes e), \mbox{ for all } e \in \mathbb{R}^{n-1}.\]
Since $R_{\infty}$ is invertible this can be written as
\begin{equation}\label{unknownoperatorfinal}
\langle e,Lx\rangle = L_{\Theta}(x \otimes (R_{\infty}^{-1}e)) \mbox{ for all } e \in \mathbb{R}^{n-1}.
\end{equation}
For a fixed $x$, the map $e\mapsto L_{\Theta}(x \otimes (R_{\infty}^{-1}e))$ defines a linear functional on $\mathbb{R}^{n-1}$ and the left hand side of \eqref{unknownoperatorfinal} says this linear functional is represented by $Lx \in \mathbb{R}^{n-1}$. Hence $Lx$ is uniquely defined and linearity of \eqref{unknownoperatorfinal} in $x$ shows that $L$ is linear. Further, $L$ is bounded as $|Lx|\leq \|L_{\Theta}\|\|R_{\infty}^{-1}\||x|$. Hence we have shown the following lemma.

\begin{lemma}\label{unique}
The system of equations \eqref{unknownoperator} uniquely defines an operator $L \in L(\mathbb{R}^{n}, \mathbb{R}^{n-1})$ with $\|L\|\leq \|L_{\Theta}\|\|R_{\infty}^{-1}\|\leq \|L_{\Theta}\|$.
\end{lemma}


Even though we do not know that $f'(x_{\infty})$ exists, some arguments leading to
\eqref{derivativeequation} can still be used to prove

\begin{lemma}\label{directionalknown}
$f'(x_{\infty},T_{\infty})=LT_{\infty}$.
\end{lemma}

\begin{proof}
Since $f$ is regularly differentiable
at $x_\infty$ in the direction of $T_\infty$, it is regularly differentiable at 
$x_\infty$ in the direction of $T$ for every $T\in L(\mathbb{R}^{n-1},V)$
where $V$ is the range of $T_\infty$. Hence, if $T\in L(\mathbb{R}^{n-1},V)$ and $\|T-T_\infty\|<\frac12\eta_0$,
we have $(x_\infty,T)\in D_0$ and the arguments above show that \eqref{derivativeequation}
holds for $S\in L(\mathbb{R}^{n-1},V)$.
Defining $\tilde L \in L(V, \mathbb{R}^{n-1})$ by $\tilde L(v)=f'(x_\infty;v)$, 
this says
\[
\langle R_{\infty},\tilde LS\rangle_{H}=L_{\Theta}S \mbox{ for all }S \in L(\mathbb{R}^{n-1}, V).
\]
Replacing $\mathbb{R}^n$ by $V$ in the proof of Lemma~\ref{unique} shows that this system has a unique solution;
since the restriction of $L$ to $V$ solves it, we get the required conclusion
$f'(x_\infty;v)=\tilde L(v)=L(v)$ for every $v\in V$.  
\end{proof}

In order to prove Proposition \ref{mainprop} it suffices to prove the following proposition.

\begin{proposition}\label{finalprop}
$f$ is differentiable at $x_{\infty} \in N$ and $f'(x_{\infty})=L$.
\end{proposition}

To prove Proposition \ref{finalprop} we argue by contradiction. Assume that $L$ is not the derivative of $f$ at $x_{\infty}$.

\begin{lemma}\label{badpoint}
There is $\varepsilon > 0$ such that for every $\delta > 0$ one may find $x \in \mathbb{R}^{n}$ with $|x|<\delta$ such that
\begin{itemize}
	\item $Lx=0$.
	\item $|f(x_{\infty}+x)-f(x_{\infty})|>\varepsilon |x|$.
\end{itemize}
\end{lemma}

\begin{proof}
Let $c=1+\|T_{\infty}L_{\infty}^{-1}L\|+\|L_{\infty}^{-1}L\|$. Since $L$ is not the derivative of $f$ at $x_{\infty}$, there is $\varepsilon > 0$ such that for every $\delta>0$ there exists $\widetilde{x} \in \mathbb{R}^{n}$ with $|\widetilde{x}|<\delta$ and
\begin{equation}\label{notderivative}
|f(x_{\infty}+\widetilde{x})-f(x_{\infty})-L\widetilde{x}|>2\varepsilon c |\widetilde{x}|.
\end{equation}

Suppose $\delta>0$ is fixed. Since $f$ is regularly differentiable at $x_{\infty}$ in direction $T_{\infty}$ there exists $\eta>0$ such that
\begin{equation}\label{reginnotderivative}
|f(x_{\infty}+z+T_{\infty}u)-f(x_{\infty}+z)-L_{\infty}u|\leq \varepsilon (|z|+|u|)
\end{equation}
whenever $z \in \mathbb{R}^{n}, u \in \mathbb{R}^{n-1}$ and $|z|+|u|<\eta$.

Choose $\widetilde{x} \in \mathbb{R}^{n}$ such that $c|\widetilde{x}|<\min \{\delta,\, \eta\}$ and \eqref{notderivative} holds. Define the point $x=\widetilde{x}-T_{\infty}L_{\infty}^{-1}L\widetilde{x}$. Since $LT_{\infty}=L_{\infty}$ by Lemma \ref{directionalknown}, it follows $Lx=0$. Further,
\[|x|+|L_{\infty}^{-1}L\widetilde{x}| \leq |\widetilde{x}|+|T_{\infty}L_{\infty}^{-1}L\widetilde{x}|+|L_{\infty}^{-1}L\widetilde{x}|\leq c|\widetilde{x}|<\min \{\delta,\, \eta \}.\]

Hence we may apply \eqref{reginnotderivative} with $z=x$ and $u=L_{\infty}^{-1}L\widetilde{x}$ to obtain
\begin{align*}
&|f(x_{\infty}+\widetilde{x})-f(x_{\infty}+x)-L\widetilde{x}|\\
&\qquad =|f(x_{\infty}+x+T_{\infty}L_{\infty}^{-1}L\widetilde{x})-f(x_{\infty}+x)-L\widetilde{x}|\\
&\qquad \leq \varepsilon(|x|+|L_{\infty}^{-1}L\widetilde{x}|) \leq \varepsilon c |\widetilde{x}|.
\end{align*}

By combining this with \eqref{notderivative} and using the triangle inequality we obtain
\begin{align*}
& |f(x_{\infty}+x)-f(x_{\infty})|\\
&\qquad \geq |f(x_{\infty}+\widetilde{x})-f(x_{\infty})-L\widetilde{x}|-|f(x_{\infty}+\widetilde{x})-f(x_{\infty}+x)-L\widetilde{x}|\\
&\qquad > \varepsilon c|\widetilde{x}| \geq \varepsilon |x|
\end{align*}
as required.
\end{proof}

We now define various parameters and a region in which to look for a pair $(x,T)\in D_{0}$ with $h_{\infty}(x,T)<h_{\infty}(x_{\infty},T_{\infty})$. We fix $0<\varepsilon <1$ with the property from Lemma \ref{badpoint}. By differentiability of $\Theta_{\infty}$ we find $K \geq 1/\varepsilon$ such that the parameter
\[\kappa=\frac{s_{K}}{48\mathrm{Lip}(f)}\]
satisfies $\kappa < \eta_{0}/4$ and
\begin{equation}\label{thetainfinity}
\Theta_{\infty}(T_{\infty}+S)-\Theta_{\infty}(T_{\infty}) \leq L_{\Theta}S + \frac{\varepsilon}{8}\|S\|
\end{equation}
whenever $\|S\| \leq \kappa$. Let $\tau=\varepsilon \kappa /4$,
\[t=\frac{\tau}{22n\mathrm{Lip}(f) (1+\|R_{\infty}\|_{H})},\]
and, using Lemma \ref{doublingfills}, choose $0<\xi \leq t/6$ such that for every $\theta>0$ we have
\[\frac{\mu(B(0,\theta(1-24\xi/ts_{K})))}{\mu(B(0,\theta))}(1-24n\xi/s_{K})\geq 1-\tau.\]

Since $f$ is regularly differentiable at each $x_{i}$ in the direction of $T_{i}$, there is $\delta_{1}>0$ such that for all $i=0, \ldots, K$ and $i=\infty$,
\begin{equation}\label{xiregdiff}
|f(x_{i}+z+T_{i}v)-f(x_{i}+z)-L_{i}v| \leq \xi (|z|+|v|)
\end{equation}
whenever $z \in \mathbb{R}^{n}, v \in \mathbb{R}^{n-1}$ and $|z|+|v|<\delta_{1}$.

Let $\delta_{2}>0$ be such that
\[\delta_{2} \leq \min \left\{ \frac{\kappa \delta_{1}}{6},\, \frac{\kappa \eta_{0}}{6},\, \frac{\delta_{1} \kappa \tau}{24\mathrm{Lip}(f)}\right\}\]
and $z \in N,\,|z-x_{\infty}|\leq 3\delta_{2}/\kappa$ implies
\[d_{N}(z,x_{\infty}) \leq \min\left\{\frac{\tau}{\sum_{i=0}^{\infty}\lambda_{i}},\, \frac{\eta_{0}}{2}\right\}.\]

Use Lemma \ref{badpoint} to find $x \in \mathbb{R}^{n}$ such that $|x|<\delta_{2}$, $Lx=0$ and
\[|f(x_{\infty}+x)-f(x_{\infty})|>\varepsilon |x|.\]

Denote
\[r=\frac{|x|}{\kappa}, \quad w=\frac{1}{r}(f(x_{\infty}+x)-f(x_{\infty})), \quad e=\frac{R_{\infty}^{-1} w}{|R_{\infty}^{-1}w|},\]
and
\[\Omega=\{u \in \mathbb{R}^{n-1}: -r<\langle e,u\rangle<0,\, |u-\langle e,u\rangle e|<r\}.\]

We have, since $\|R_{\infty}^{-1}\|\leq 1$,
\[\langle R_{\infty}e,w\rangle=\frac{|w|^{2}}{|R_{\infty}^{-1}w|} \geq |w| >\varepsilon \kappa = 4 \tau.\]

Temporarily fix $\widetilde{e} \in \mathbb{R}^{n-1}$, $\widetilde{x} \in \mathbb{R}^{n}$ and $\widetilde{T}_{\infty} \in L(\mathbb{R}^{n-1}, \mathbb{R}^{n})$ close to $e$, $x$ and $T_{\infty}$ respectively. Define an affine map $\gamma\colon \mathbb{R}^{n-1} \to \mathbb{R}^{n}$ with $\mathrm{Lip}(\gamma) \leq 2 \kappa$ by
\begin{equation}\label{defgamma}\gamma(u) =\widetilde{x}+\langle \widetilde{e}, u \rangle \widetilde{x}/r\end{equation}
and let $\varphi\colon \mathbb{R}^{n} \to \mathbb{R}^{n}$ be an affine change of coordinates given by
\begin{equation}\label{defvarphi}\varphi(u)=x_{\infty}+\gamma(\pi_{n-1}u)+\widetilde{T}_{\infty}\pi_{n-1}u+(0,\pi^{n-1}u)\end{equation}
where $\pi_{n-1}$ and $\pi^{n-1}$ are the orthogonal projections onto the first $n$ coordinates and the final coordinate respectively. That is, for $a \in \mathbb{R}^{n-1}$ and $b \in \mathbb{R}$, we have $\pi_{n-1}(a,b)=a$ and $\pi^{n-1}(a,b)=b$.

Given $u \in \mathbb{R}^{n}$ we define $\varphi_{n-1}'(u) \in L(\mathbb{R}^{n-1},\mathbb{R}^{n})$ to be the derivative, at $\pi_{n-1}(u)$ of the map $\mathbb{R}^{n-1} \to \mathbb{R}^{n}$ given by $v\mapsto \varphi(v,\pi^{n-1}(u))$.

Note that,
\[\varphi(u)=x_{\infty}+\widetilde{x} + \kappa \langle \widetilde{e},u\rangle \frac{\widetilde{x}}{|x|}+\widetilde{T}_{\infty}\pi_{n-1}u+(0,\pi^{n-1}u)\]
and
\[\varphi_{n-1}'(u)=\kappa \frac{\widetilde{x}\otimes \widetilde{e}}{|x|}+\widetilde{T}_{\infty}\]
for $u \in \mathbb{R}^{n}$. Let $Q=\Omega \times (-tr,tr)$.

\begin{lemma}\label{goodplane}
Let $\alpha>0$. Then there exists $\widetilde{e} \in \mathbb{R}^{n-1}$, $\widetilde{x} \in \mathbb{R}^{n}$, and a linear map $\widetilde{T}_{\infty} \in L(\mathbb{R}^{n-1}, \mathbb{R}^{n})$ such that
\[|\widetilde{e}-e|+|\widetilde{x}-x|+\|\widetilde{T}_{\infty}-T_{\infty}\|<\alpha\]
and, for $\mathcal{L}^{n-1} \times \mu$ almost every $u \in Q$,
\begin{itemize}
	\item The function $f$ is regularly differentiable at $\varphi(u)$ in the direction $\varphi_{n-1}'(u)$,
	\item $\varphi(u) \in N$.
\end{itemize}
\end{lemma}

\begin{proof}
We can choose $\widetilde{e}, \widetilde{x}$ and $\widetilde{T}_{\infty}$ with
\[|\widetilde{e}-e|+|\widetilde{x}-x|+\|\widetilde{T}_{\infty}-T_{\infty}\|<\alpha\]
so that $\varphi$ is bilipschitz and belongs to the countable dense family of affine maps in Definition \ref{nullset}. This implies $\varphi(u) \in N$ whenever $\pi^{n-1}(u) \in \widetilde{N}$. Hence, since $\mu(\mathbb{R}\setminus \widetilde{N})=0$, it follows $\varphi(u) \in N$ for $\mathcal{L}^{n-1} \times \mu$ almost every $u \in Q$.

The preimage, under the bilipschitz map $\varphi$, of the $\sigma$-porous set of points at which $f$ is not regular is again $\sigma$-porous. Hence, since $\mathcal{L}^{n-1} \times \mu$ is doubling, Lemma \ref{doublingporousnull} and Lemma \ref{irregularporous} imply that the function $f$ is regular at $\varphi(u)$ for $\mathcal{L}^{n-1} \times \mu$ almost every $u$.

By the classical Rademacher theorem we know that for each fixed $b \in \mathbb{R}$, $f$ is differentiable at $\varphi(a,b)$ in direction $\varphi_{n-1}'(a,b)$ for $\mathcal{L}^{n-1}$ almost every $a \in \mathbb{R}^{n-1}$. Hence $f$ is differentiable at $\varphi(u)$ in direction $\varphi_{n-1}'(u)$ for $\mathcal{L}^{n-1} \times \mu$ almost every $u \in Q$. By Lemma \ref{regularpointdiff} this proves the result.
\end{proof}

We now fix $\alpha>0$ small relative to all previous parameters. Since $\alpha$ is the last parameter we define, we don't list precise estimates. When the fact $\alpha$ is small is used, it should be clear $\alpha$ could have been chosen appropriately at this stage. Fix $\widetilde{e} \in \mathbb{R}^{n-1}$, $\widetilde{x} \in \mathbb{R}^{n}$ and $\widetilde{T}_{\infty} \in L(\mathbb{R}^{n-1}, \mathbb{R}^{n})$ as in Lemma \ref{goodplane} with the corresponding $\gamma$ and $\varphi$ as defined in \eqref{defgamma} and \eqref{defvarphi}. Also define $g\colon \mathbb{R}^{n} \to \mathbb{R}^{n-1}$ by
\[g(u)=f(\varphi(u))-L_{\infty}\pi_{n-1}u.\]

\begin{lemma}\label{phigestimates}
If $\alpha$ is sufficiently small then $\mathrm{Lip}(\varphi) \leq 3$, $\mathrm{Lip}(g) \leq 5\mathrm{Lip}(f)$,
\begin{itemize}
	\item $|\varphi(u)-x_{\infty}|\leq 3r$ for all $u \in Q$, and
	\item for every $u, v \in \overline{Q}$ with $\langle u-v,e\rangle =0$,
	\[|g(u)-g(v)|\leq \frac{\tau r}{n(1+\|R_{\infty}\|_{H})}.\]
\end{itemize}
\end{lemma}

\begin{proof}
We use Lemma \ref{dtlestimatesapplied}. For the first inequality, since $\|T_{\infty}\|\leq 1$ and $\alpha$ is small,
\[\mathrm{Lip}(\varphi)\leq 2\kappa + \|\widetilde{T}_{\infty}\|+1 \leq 3.\]

For the second inequality, the facts $\|L_{\infty}\|\leq 5/4$ and $\mathrm{Lip}(f)\geq 1$ immediately imply
\[\mathrm{Lip}(g) \leq \frac{5}{4}+3\mathrm{Lip}(f) \leq 5\mathrm{Lip}(f).\]
For the third inequality, for small $\alpha$, we estimate,
\[|\varphi(u)-x_{\infty}| \leq \kappa r + 2\kappa r + 2r + tr + \kappa r = (4\kappa + t+2)r \leq 3r.\]

For the final estimate we use regular differentiability of $f$ at $x_{\infty}$ in direction $T_{\infty}$ which is \eqref{xiregdiff}. For convenience, first assume $\widetilde{e}=e, \widetilde{x}=x$ and $\widetilde{T}_{\infty}=T_{\infty}$.

We next suppose $u, v \in \overline{\Omega}\subset \mathbb{R}^{n-1} \subset \mathbb{R}^{n}$. Since $\langle u-v, e \rangle=0$ it follows that $\gamma(\pi_{n-1}u)=\gamma(\pi_{n-1}v)$ and hence $\varphi(u)-\varphi(v)=T_{\infty}(u-v)$. We use \eqref{xiregdiff} with $z=\gamma(v)+T_{\infty}v$ and $u-v$ in place of $u$. We estimate,
\[|\gamma(v)+T_{\infty}v|+|u-v|\leq |\varphi(v)-x_{\infty}|+|u-v|\leq 6r < \delta_{1}.\]
Hence, using \eqref{xiregdiff} and recalling $\gamma(u)=\gamma(v)$,
\begin{align*}
&|g(u)-g(v)|\\
&=|f(x_{\infty}+\gamma(u)+T_{\infty}u)-f(x_{\infty}+\gamma(v)+T_{\infty}v)-L_{\infty}(u-v)|\\
&= |f(x_{\infty}+\gamma(v)+T_{\infty}v+T_{\infty}(u-v))-f(x_{\infty}+\gamma(v)+T_{\infty}v)-L_{\infty}(u-v)|\\
&\leq \xi(|\gamma(v)+T_{\infty}v|+|u-v|)\\
&\leq 6\xi r.
\end{align*}
For general $u, v \in \overline{Q}$ we just use that $Q$ is relatively thin in the remaining direction,
\begin{align*}
|g(u)-g(v)|&\leq |g(u)-g(\pi_{n-1} u)|+|g(v)-g(\pi_{n-1} v)|\\
& \qquad +|g(\pi_{n-1} u)-g(\pi_{n-1} v)|\\
&\leq 2\mathrm{Lip}(g)tr + 6\xi r \leq 10\mathrm{Lip}(f)tr+tr\\
&\leq 11\mathrm{Lip}(f)tr=\frac{\tau r}{2n(1+\|R_{\infty}\|_{H})}.
\end{align*}
For general $\widetilde{e}, \widetilde{x}$ and $\widetilde{T}_{\infty}$ we note that, if we temporarily denote $g=g_{\widetilde{e},\widetilde{x},\widetilde{T}_{\infty}}$ then since $f$ is Lipschitz, provided $\alpha$ is sufficiently small, we can ensure
\[|g_{\widetilde{e},\widetilde{x},\widetilde{T}_{\infty}}(u)-g_{e,x,T_{\infty}}(u)|\leq \frac{\tau r}{4n(1+\|R_{\infty}\|_{H})}\]
for all $u \in \overline{Q}$. Thus, provided $\alpha$ is sufficiently small, the result follows.
\end{proof}

We now observe that $(\varphi(u), \varphi_{n-1}'(u)) \in D_{0}$ for $\mathcal{L}^{n-1} \times \mu$ almost every $u \in Q$. Indeed, for all $u \in Q$ we have $|\varphi(u)-x_{\infty}|\leq 3r$. If $\varphi(u) \in N$ this implies $d_{N}(\varphi(u),x_{\infty})\leq \eta_{0}/2$. Hence, since $d_{N}(x_{\infty}, x_{0})< \eta_{0}/2$ by Lemma \ref{dtlestimatesapplied},
\[d_{N}(\varphi(u),x_{0})\leq d_{N}(\varphi(u),x_{\infty})+d_{N}(x_{\infty},x_{0}) \leq \eta_{0}.\]
Further, using the expression for $\varphi_{n-1}'$ and Lemma \ref{dtlestimatesapplied}, it follows that for sufficiently small $\alpha$,
\[\|\varphi_{n-1}'(u)-T_{0}\| \leq \|\varphi_{n-1}'(u)-\widetilde{T}_{\infty}\|+\|\widetilde{T}_{\infty}-T_{\infty}\|+\|T_{\infty}-T_{0}\| \leq \eta_{0}.\]
By our choice of $\varphi$ after Lemma \ref{goodplane}, $\varphi(u) \in N$ and $f$ is regularly differentiable at $\varphi(u)$ in direction $\varphi_{n-1}'(u)$ at $\mathcal{L}^{n-1} \times \mu$ almost every $u\in Q$. Hence $(\varphi(u),\varphi_{n-1}'(u)) \in D_{0}$ for $\mathcal{L}^{n-1} \times \mu$ almost every $u \in Q$.

We aim to show that
\begin{equation}\label{positive}
\dashint_{Q} (h_{\infty}(x_{\infty},T_{\infty})-h_{\infty}(\varphi,\varphi_{n-1}')) \dd(\mathcal{L}^{n-1} \times \mu) > 0.
\end{equation}
Once this is done a contradiction follows. Indeed, suppose \eqref{positive} holds. Then there is $u \in Q$ such that $(\varphi(u),\varphi_{n-1}'(u)) \in D_{0}$ and 
\[h_{\infty}(x_{\infty},T_{\infty})-h_{\infty}(\varphi(u),\varphi_{n-1}'(u))>0\]
which contradicts the assumption $(x_{\infty}, T_{\infty})$ is a minimizer of $h_{\infty}$ in $D_{0}$.

\section{Integral Estimates}

We now prove \eqref{positive}. Recall from \eqref{changeinh},
\begin{equation}\label{changeinhwithvarphi}\begin{split}
&h_{\infty}(x_{\infty},T_{\infty})-h_{\infty}(\varphi,\varphi_{n-1}')\\
&\qquad =(\|f'(\varphi;\varphi_{n-1}')\|_{H}^{2}-\|L_{\infty}\|_{H}^{2})\\
&\qquad \qquad -\Phi(\varphi)-\Psi(\varphi_{n-1}')-\Upsilon(f'(\varphi;\varphi_{n-1}'))-\Delta(\varphi).
\end{split}\end{equation}
We now estimate the average integral of each of the terms on the right side. Most of the estimates are very similar to those in \citep{LPT12}. The main difference is in the estimate of the regularity term. Here we need to account for the fact $\mu$ is not Lebesgue measure but can still make the necessary estimates using Proposition \ref{integralregularity}.

\subsection{Estimate of $\dashint_{Q} \Phi(\varphi) \dd(\mathcal{L}^{n-1} \times \mu)$}

For $\mathcal{L}^{n-1} \times \mu$ almost every $u \in Q$ we know, by Lemma $\ref{goodplane}$, $\varphi(u) \in N$. By Lemma \ref{phigestimates} we know $|\varphi(u)-x_{\infty}|\leq 3r$ for every $u \in \overline{Q}$. By our choice of $r$ this implies $d_{N}(\varphi(u),x_{\infty}) \leq \tau/(\sum_{i=0}^{\infty} \lambda_{i})$ for $\mathcal{L}^{n-1} \times \mu$ almost every $u \in Q$

Hence, for such $u$,
\[|d_{N}(x_{i},\varphi(u))-d_{N}(x_{i},x_{\infty})|\leq d_{N}(\varphi(u),x_{\infty})\leq \tau/\sum_{i=0}^{\infty}\lambda_{i}.\]
This implies,
\[|\Phi(\varphi(u))|\leq \sum_{i=0}^{\infty} \lambda_{i} d_{N}(\varphi(u),x_{\infty}) \leq \tau.\]
Hence
\[\dashint_{Q} \Phi(\varphi) \dd(\mathcal{L}^{n-1} \times \mu) \leq \tau.\]

\subsection{Estimate of $\dashint_{Q} \Psi(\varphi_{n-1}') \dd(\mathcal{L}^{n-1} \times \mu)$}

By our choice of $x$ we know $Lx=0$. By the definition of $L$ this implies $L_{\Theta}(x \otimes e)=0$. Using \eqref{thetainfinity} we obtain, provided $\alpha$ is sufficiently small,
\begin{align*}
\Psi(\varphi_{n-1}')&=\Theta_{\infty}\left(\widetilde{T}_{\infty}+\kappa\frac{\widetilde{x} \otimes \widetilde{e}}{|x|}\right) - \Theta_{\infty}\left(T_{\infty}\right)\\
&\leq \Theta_{\infty}\left(T_{\infty}+\kappa\frac{x \otimes e}{|x|}\right) - \Theta_{\infty}\left(T_{\infty}\right)\\
& \qquad + \mathrm{Lip}(\Theta_{\infty})\left\|\widetilde{T}_{\infty} + \kappa\frac{\widetilde{x} \otimes \widetilde{e}}{|x|} - T_{\infty}-\kappa\frac{x \otimes e}{|x|}\right\|\\
&\leq \frac{\varepsilon \kappa}{8}+\frac{\varepsilon \kappa}{8}\\
&= \tau.
\end{align*}
Hence
\[\dashint_{Q} \Psi(\varphi_{n-1}') \dd(\mathcal{L}^{n-1} \times \mu) \leq \tau.\]

\subsection{Estimate of $\dashint_{Q} \Upsilon(f'(\varphi;\varphi_{n-1}')) \dd(\mathcal{L}^{n-1} \times \mu)$}

Using the definition of $g$ we have $f'(\varphi;\varphi_{n-1}')=g_{n-1}'+L_{\infty}$. Since $\sum_{i=0}^{\infty} \gamma_{i}=1/4$ we have,
\begin{align*}
\Upsilon(f'(\varphi(u);\varphi_{n-1}'(u))) &= \sum_{i=0}^{\infty} \gamma_{i}(\|(L_{i}-L_{\infty})-g_{n-1}'(u)\|_{H}^{2}-\|L_{i}-L_{\infty}\|_{H}^{2})\\
&=\sum_{i=0}^{\infty} \gamma_{i}(2\langle(L_{\infty}-L_{i}),g_{n-1}'(u)\rangle_{H}+\|g_{n-1}'(u)\|_{H}^{2})\\
&=\left\langle 2\sum_{i=0}^{\infty} \gamma_{i}(L_{\infty}-L_{i}),g_{n-1}'(u)\right\rangle_{H}+\frac{1}{4}\|g_{n-1}'(u)\|_{H}^{2}.
\end{align*}
Hence 
\begin{align*}
&\dashint_{Q} \Upsilon(f'(\varphi;\varphi_{n-1}'))\dd(\mathcal{L}^{n-1} \times \mu)\\ 
& \qquad = \dashint_{Q} \left\langle 2\sum_{i=0}^{\infty} \gamma_{i}(L_{\infty}-L_{i}),g_{n-1}'(u)\right\rangle_{H} \dd(\mathcal{L}^{n-1} \times \mu)\\
&\qquad \qquad + \frac{1}{4}\dashint_{Q} \|g_{n-1}'\|_{H}^{2}\dd(\mathcal{L}^{n-1} \times \mu).
\end{align*}

\subsection{Estimate of $\dashint_{Q} \Delta(\varphi) \dd(\mathcal{L}^{n-1} \times \mu)$}

First fix $0 \leq i \leq K$. Denote $s=24\xi r/s_{K}$ and
\[P=\{u \in Q: \Delta_{i}((x_{i},T_{i}),(\varphi(u),\varphi_{n-1}'(u)))-\Delta_{i}((x_{i},T_{i}),(x_{\infty},T_{\infty}))> \tau\}.\]

We show that for every $u \in P$,
\begin{equation}\label{regsi}
\mathrm{reg}_{n-1}g(u,B(u,s))>\frac{s_{i}}{8}.
\end{equation}

Fix $u \in P$. Choose $\widetilde{z} \in \mathbb{R}^{n}$ and $\widetilde{w} \in \mathbb{R}^{n-1} \subset \mathbb{R}^{n}$ such that
\begin{align*}
&|f_{x_{i},T_{i}}(\widetilde{z},\widetilde{w})-f_{x_{i},T_{i}}(\widetilde{z},0)-f_{\varphi(u),T_{i}}(\widetilde{z},\widetilde{w})+f_{\varphi(u),T_{i}}(\widetilde{z},0)|\\
&\qquad > \left(\|f_{x_{i},T_{i}}-f_{\varphi(u),T_{i}}\|_{r}-\frac{\tau}{2}\right)(|\widetilde{z}|+|\widetilde{w}|).
\end{align*}
We show that if $\widetilde{u}=\widetilde{z}+\widetilde{w}$ then
\begin{equation}\label{regsilemma}
|\widetilde{z}|+|\widetilde{w}|<\delta_{1} \mbox{ and } u+\widetilde{u}, u+\widetilde{z} \in B(u,s).
\end{equation}

We have two cases

\begin{case}
Suppose that $|\widetilde{z}|+|\widetilde{w}|\geq \delta_{1}/2$. From Lemma \ref{phigestimates} and the choice of $\delta_{2}$,
\[2\mathrm{Lip}(f)|\varphi(u)-x_{\infty}| \leq 6\mathrm{Lip}(f) r \leq 6\mathrm{Lip}(f)\frac{\delta_{2}}{\kappa}\leq \frac{\tau}{4}\delta_{1}\leq \frac{\tau}{2}(|\widetilde{z}|+|\widetilde{w}|).\]
It follows that
\begin{align*}
&(|\widetilde{z}|+|\widetilde{w}|)\|f_{x_{i},T_{i}}-f_{x_{\infty},T_{i}}\|_{r}\\
&\qquad \geq |f_{x_{i},T_{i}}(\widetilde{z},\widetilde{w}) - f_{x_{i},T_{i}}(\widetilde{z},0) - f_{x_{\infty},T_{i}}(\widetilde{z},\widetilde{w}) + f_{x_{\infty},T_{i}}(\widetilde{z},0)|\\
&\qquad \geq |f_{x_{i},T_{i}}(\widetilde{z},\widetilde{w}) - f_{x_{i},T_{i}}(\widetilde{z},0) - f_{\varphi(u),T_{i}}(\widetilde{z},\widetilde{w}) + f_{\varphi(u),T_{i}}(\widetilde{z},0)|\\
&\qquad \qquad -2\mathrm{Lip}(f)|\varphi(u)-x_{\infty}|\\
&\qquad \geq \left(\|f_{x_{i},T_{i}}-f_{\varphi(u),T_{i}}\|_{r}-\frac{\tau}{2}\right)(|\widetilde{z}|+|\widetilde{w}|)-\frac{\tau}{2}(|\widetilde{z}|+|\widetilde{w}|)\\
&\qquad \geq (|\widetilde{z}|+|\widetilde{w}|)(\|f_{x_{i},T_{i}}-f_{\varphi(u),T_{i}}\|_{r}-\tau).
\end{align*}
Using the definition of $\Delta_{i}$ this implies
\[\Delta_{i}((x_{i},T_{i}),(\varphi(u),\varphi_{n-1}'(u)))-\Delta_{i}((x_{i},T_{i}),(x_{\infty},T_{\infty}))\leq \tau\]
which contradicts the assumption that $u \in P$.
\end{case}

\begin{case}
Suppose that $|\widetilde{z}|+|\widetilde{w}|<\delta_{1}/2$. If $s \geq |\widetilde{u}|+|\widetilde{z}|$ then \eqref{regsilemma} is clear. Hence we suppose $s\leq |\widetilde{u}|+|\widetilde{z}|$. It follows that $s \leq 2(|\widetilde{z}|+|\widetilde{w}|)$. Define $\widehat{z} \in \mathbb{R}^{n}$ by the requirement that $\varphi(u)+\widetilde{z}=x_{\infty}+\widehat{z}$. Then by Lemma \ref{phigestimates} and the choice of $\delta_{2}$,
\[|\widehat{z}|+|\widetilde{w}|\leq |\varphi(u)-x_{\infty}|+|\widetilde{z}|+|\widetilde{w}|\leq 3r+|\widetilde{z}|+|\widetilde{w}|<\delta_{1}.\]
We use the estimate \eqref{xiregdiff} from regular differentiability, our choice of $\xi$ and $s$, and the fact $\|T_{i}-T_{\infty}\|\leq s_{i}/8\mathrm{Lip}(f)$ and $\|L_{i}-L_{\infty}\|\leq s_{i}/8$ from Lemma \ref{dtlestimatesapplied} to conclude,

\begin{align*}
|f_{\varphi(u),T_{i}}(\widetilde{z},\widetilde{w})&-f_{\varphi(u),T_{i}}(\widetilde{z},0)-L_{i}\widetilde{w}|\\
&=|f_{x_{\infty},T_{i}}(\widehat{z},\widetilde{w})-f_{x_{\infty},T_{i}}(\widehat{z},0)-L_{i}\widetilde{w}|\\
&\leq |f_{x_{\infty},T_{\infty}}(\widehat{z},\widetilde{w})-f_{x_{\infty},T_{\infty}}(\widehat{z},0)-L_{\infty}\widetilde{w}|\\
& \qquad + |f_{x_{\infty},T_{\infty}}(\widehat{z},\widetilde{w})-f_{x_{\infty},T_{i}}(\widehat{z},\widetilde{w})| + |L_{i}\widetilde{w}-L_{\infty}\widetilde{w}|\\
&\leq \xi(|\widehat{z}|+|\widetilde{w}|)+\mathrm{Lip}(f)|T_{\infty}\widetilde{w}-T_{i}\widetilde{w}|+\frac{s_{i}}{8}|\widetilde{w}|\\
&\leq \xi(3r+|\widetilde{z}|+|\widetilde{w}|)+\frac{s_{i}}{4}|\widetilde{w}|\\
&\leq \frac{ss_{i}}{8}+\frac{s_{i}}{2}(|\widetilde{z}|+|\widetilde{w}|)\\
&\leq \frac{3s_{i}}{4}(|\widetilde{z}|+|\widetilde{w}|).
\end{align*}

Since $|\widetilde{z}|+|\widetilde{w}|<\delta_{1}$, another application of regular differentiability in \eqref{xiregdiff} implies
\[|f_{x_{i},T_{i}}(\widetilde{z},\widetilde{w})-f_{x_{i},T_{i}}(\widetilde{z},0)-L_{i}\widetilde{w}|\leq \xi(|\widetilde{z}|+|\widetilde{w}|)\leq \frac{s_{i}}{4}(|\widetilde{z}|+|\widetilde{w}|).\]
Hence we obtain, by triangle inequality and the original choice of $\widetilde{z}$ and $\widetilde{w}$,
\[\left(\|f_{x_{i},T_{i}}-f_{\varphi(u),T_{i}}\|_{r}-\frac{\tau}{2}\right)(|\widetilde{z}|+|\widetilde{w}|) \leq s_{i}(|\widetilde{z}|+|\widetilde{w}|)\]
which implies $\Delta_{i}((x_{i},T_{i}),(\varphi(u),\varphi_{n-1}'(u)))\leq \tau /2$ so $u \notin P$.
\end{case}

We have established \eqref{regsilemma} and now show \eqref{regsi}. We first estimate, using the fact $\alpha$ is small and $\widetilde{w} \in \mathbb{R}^{n-1}$,
\begin{align*}
&|\varphi(u+\widetilde{w})-\varphi(u)-T_{i}\widetilde{w}|\\
& \qquad \leq (\|T_{i}-T_{\infty}\|+\|T_{\infty}-\widetilde{T}_{\infty}\|+\mathrm{Lip}(\gamma))|\widetilde{w}|\\
& \qquad \leq \left( \frac{s_{i}}{8\mathrm{Lip}(f)}+\kappa+2\kappa \right)|\widetilde{w}|\\
& \qquad \leq \frac{s_{i}}{4\mathrm{Lip}(f)}|\widetilde{w}|.
\end{align*}

Combining this with \eqref{regsilemma},
\begin{align*}
(|\widetilde{u}|&+|\widetilde{z}|)\mathrm{reg}_{n-1}g(u,B(u,s)) \geq |g(u+\widetilde{u})-g(u+\widetilde{z})|\\
&=|f(\varphi((u+\widetilde{z}+\widetilde{w})))-f(\varphi((u+\widetilde{z})))-L_{\infty}(\widetilde{w})|\\
&\geq |f(\varphi(u+\widetilde{w})+\widetilde{z})-f(\varphi(u)+\widetilde{z})-L_{\infty}(\widetilde{w})|\\
&\qquad - |f(\varphi(u+\widetilde{w})+\widetilde{z})-f(\varphi(u)+\widetilde{z})\\
&\qquad \qquad -f(\varphi((u+\widetilde{z}+\widetilde{w})))+f(\varphi((u+\widetilde{z})))|\\
&\geq |f_{\varphi(u),T_{i}}(\widetilde{z},\widetilde{w})-f_{\varphi(u),T_{i}}(\widetilde{z},0)-L_{\infty}(\widetilde{w})|\\
&\qquad - \mathrm{Lip}(f)|\varphi(u+\widetilde{w}) -\varphi(u)-T_{i}\widetilde{w}|\\
&\qquad \qquad -2\mathrm{Lip}(f)\mathrm{Lip}(\varphi) |\widetilde{w}|\\
&\geq |f_{\varphi(u),T_{i}}(\widetilde{z},\widetilde{w})-f_{\varphi(u),T_{i}}(\widetilde{z},0)+f_{x_{i},T_{i}}(\widetilde{z},0)-f_{x_{i},T_{i}}(\widetilde{z},\widetilde{w})|\\
&\qquad -|f_{x_{i},T_{i}}(\widetilde{z},\widetilde{w})-f_{x_{i},T_{i}}(\widetilde{z},0)-L_{i}(\widetilde{w})|-\|L_{i}-L_{\infty}\| |\widetilde{w}|\\
&\qquad \qquad -\frac{s_{i}}{4}|\widetilde{w}|-6\mathrm{Lip}(f)\kappa |\widetilde{w}|\\
&\geq \left( \|f_{x_{i},T_{i}}-f_{\varphi(u),T_{i}}\|_{r}-\frac{\tau}{2}\right)(|\widetilde{z}|+|\widetilde{w}|)\\
&\qquad  -|f_{x_{i},T_{i}}(\widetilde{z},\widetilde{w})-f_{x_{i},T_{i}}(\widetilde{z},0)-L_{i}(\widetilde{w})|-\frac{s_{i}}{8} |\widetilde{w}|\\
&\qquad \qquad -\frac{s_{i}}{4}|\widetilde{w}|-\frac{s_{i}}{8} |\widetilde{w}|.
\end{align*}
For the first term we use the fact $u \in P$ implies $\|f_{x_{i},T_{i}}-f_{\varphi(u),T_{i}}\|_{r}>\tau+s_{i}$. For the second term we use regular differentiability \eqref{xiregdiff}. Hence the above expression is greater or equal than
\begin{align*}
s_{i}(|\widetilde{z}|+|\widetilde{w}|)-\xi(|\widetilde{z}|+|\widetilde{w}|)-\frac{s_{i}}{2}|\widetilde{w}| &\geq \frac{s_{i}}{4}(|\widetilde{z}|+|\widetilde{w}|)\\
& \geq \frac{s_{i}}{8}(|\widetilde{u}|+|\widetilde{z}|)
\end{align*}
which proves \eqref{regsi}.

Let $Q_{0}=\{u \in Q: B(u,s) \subset Q\}$. Then, from Proposition \ref{integralregularity},
\[ (\mathcal{L}^{n-1}\times \mu)(P \cap Q_{0}) \leq \frac{C(\mu)(5\mathrm{Lip}(f))^{2C(\mu)-2}}{(s_{i}/8)^{2C(\mu)}}\int_{Q} \|g_{n-1}'\|^{2} \dd (\mathcal{L}^{n-1} \times \mu).\]

Recall $\Omega$ is a cylinder of height $r$ whose cross sections are balls in $\mathbb{R}^{n-2}$ of radius $r$; hence $\mathcal{L}^{n-1}(\Omega)=\omega_{n-2}r^{n-1}$. We estimate, using the Bernoulli inequality and the choice of $s$ and $\xi$,

\begin{align*}
\frac{(\mathcal{L}^{n-1} \times \mu)(Q_{0})}{(\mathcal{L}^{n-1} \times \mu)(Q)} &\geq \left(1-\frac{2s}{r}\right) \left(1-\frac{s}{r}\right)^{n-2} \frac{\mu(B(0,tr-s))}{\mu(B(0,tr))}\\
&\geq \left(1-\frac{ns}{r}\right)\frac{\mu(B(0,tr-s))}{\mu(B(0,tr))}\\
&\geq (1-\tau).
\end{align*}
Using the fact $\Delta_{i}((x_{i},T_{i}),(\varphi,\varphi_{n-1}'))\leq 1$ and our choice of $\sigma_{i}$ we conclude,
\begin{align*}
&\sigma_{i}\int_{Q}( \Delta_{i}((x_{i},T_{i}),(\varphi,\varphi_{n-1}'))-\Delta_{i}((x_{i},T_{i}),(x_{\infty},T_{\infty}))) \dd (\mathcal{L}^{n-1} \times \mu)\\
&\qquad \leq \sigma_{i} \int_{P} \Delta_{i}((x_{i},T_{i}),(\varphi,\varphi_{n-1}')) \dd (\mathcal{L}^{n-1} \times \mu) + \sigma_{i}\tau (\mathcal{L}^{n-1} \times \mu)(Q \setminus P)\\
&\qquad \leq \sigma_{i}(\mathcal{L}^{n-1} \times \mu)(P \cap Q_{0}) + \sigma_{i} (\mathcal{L}^{n-1} \times \mu)(Q \setminus Q_{0})+\sigma_{i}\tau (\mathcal{L}^{n-1} \times \mu)(Q)\\
&\qquad \leq \sigma_{i}\frac{C(\mu)(5\mathrm{Lip}(f))^{2C(\mu)-2}}{(s_{i}/8)^{2C(\mu)}}\int_{Q} \|g_{n-1}'\|^{2} \dd (\mathcal{L}^{n-1} \times \mu)\\
&\qquad \qquad +2\sigma_{i}\tau(\mathcal{L}^{n-1} \times \mu)(Q)\\
&\qquad \leq 2^{-i-3}\left(\int_{Q}\|g_{n-1}'\|^{2} \dd (\mathcal{L}^{n-1} \times \mu)+\tau(\mathcal{L}^{n-1} \times \mu)(Q)\right).
\end{align*}

For $i>K$, we note that $i>K$ implies $\varepsilon>1/i$ and the choice of $\sigma_{i}, \kappa$ and $\tau$ give,
\[\sigma_{i}\leq \frac{2^{-i-5}}{i+1}\frac{s_{i}}{48\mathrm{Lip}(f)}\leq 2^{-i-3} \frac{\varepsilon \kappa}{4}=2^{-i-3}\tau.\]
Hence we obtain, for $i>K$,
\begin{align*}
&\sigma_{i}\int_{Q}( \Delta_{i}((x_{i},T_{i}),(\varphi,\varphi_{n-1}'))-\Delta_{i}((x_{i},T_{i}),(x_{\infty},T_{\infty}))) \dd (\mathcal{L}^{n-1} \times \mu)\\
&\qquad \leq 2^{-i-3}\tau (\mathcal{L}^{n-1} \times \mu)(Q).
\end{align*}

Adding the inequalities together, for different $i$, we obtain, 
\[\dashint_{Q} \Delta(\varphi) \dd(\mathcal{L}^{n-1} \times \mu) \leq \frac{1}{4}\dashint_{Q} \|g_{n-1}'\|^{2} \dd(\mathcal{L}^{n-1} \times \mu) + \frac{\tau}{4}.\]

\subsection{Combining the Estimates}

It remains to analyze the value of the first term in \eqref{changeinhwithvarphi}. We use the fact $f'(\varphi;\varphi_{n-1}')=g_{n-1}'+L_{\infty}$ to obtain,
\begin{align*}
& \dashint_{Q} (\|f'(\varphi;\varphi_{n-1}')\|_{H}^{2}-\|L_{\infty}\|_{H}^{2}) \dd(\mathcal{L}^{n-1} \times \mu)\\
& \qquad = \dashint_{Q} \|g_{n-1}'\|_{H}^{2} \dd(\mathcal{L}^{n-1} \times \mu) + \dashint_{Q} \langle 2L_{\infty},g_{n-1}'\rangle_{H} \dd(\mathcal{L}^{n-1} \times \mu).
\end{align*}
By putting together the previous estimates of this section and using the definition of $R_{\infty}$ we find (all integrals are with respect to $\mathcal{L}^{n-1} \times \mu$),
\begin{align*}
&\dashint_{Q} (h_{\infty}(x_{\infty},T_{\infty})-h_{\infty}(\varphi,\varphi_{n-1}'))\\
& \qquad \geq \dashint_{Q} \|g_{n-1}'\|_{H}^{2} + \dashint_{Q} <2L_{\infty},g_{n-1}'>_{H}-\frac{1}{4}\dashint_{Q}\|g_{n-1}'\|_{H}^{2}\\
& \qquad \qquad + \dashint_{Q} \left\langle2\sum_{i=0}^{\infty} \gamma_{i}(L_{i}-L_{\infty}),g_{n-1}'\right\rangle_{H}-\frac{1}{4}\dashint_{Q} \|g_{n-1}'\|_{H}^{2}-\frac{9\tau}{4}\\
& \qquad \geq \dashint_{Q} \langle R_{\infty},g_{n-1}'\rangle-\frac{9\tau}{4}\\
& \qquad = \langle R_{\infty}e,w\rangle + \dashint_{Q}\langle R_{\infty},g_{n-1}'-w \otimes e\rangle_{H} - \frac{9\tau}{4}\\
& \qquad \geq \frac{7\tau}{4} - \|R_{\infty}\|_{H} \left\| \dashint_{Q} (g_{n-1}'-w\otimes e)\right\|_{H}.
\end{align*}
Hence to conclude our proof it suffices to show that
\[\|R_{\infty}\|_{H}\left\| \dashint_{Q} (g_{n-1}'-w \otimes e) \dd(\mathcal{L}^{n-1} \times \mu) \right\|_{H} < \frac{7\tau}{4}.\]

\subsection{Estimate of $\left\| \dashint_{Q} (g_{n-1}'-w \otimes e) \dd(\mathcal{L}^{n-1} \times \mu) \right\|_{H}$}

We define the maps $\zeta, \widetilde{\zeta}\colon \mathbb{R}^{n} \to \mathbb{R}^{n-1}$ by
\[\zeta(u)=g(\langle u,e\rangle e) \mbox{ and } \widetilde{\zeta}(u)=\zeta(u)-g(0)-\langle u,e\rangle w.\]
By Lemma \ref{phigestimates} we have, for every $u \in \overline{Q}$,
\[|g(u)-\zeta(u)|\leq \frac{\tau r}{n(1+\|R_{\infty}\|_{H})}.\]

We now use the following inequality \citep[Corollary 9.4.2]{LPT12} which arises from an application of the divergence theorem.

\begin{lemma}
Let $\Pi \subset \mathbb{R}^{n-1}$ be a bounded open set with Lipschitz boundary. Then for every Lipschitz function $\Lambda \colon \overline{\Pi} \to \mathbb{R}^{m}$,
\[\left\| \int_{\Pi} \Lambda'(u)\dd \mathcal{L}^{n-1}(u)\right\|_{H} \leq \mathcal{H}^{n-2}(\partial \Pi) \max_{u \in \partial \Pi} |\Lambda(u)|.\]
\end{lemma}

For each $b \in (-tr,tr)$ we apply the lemma to the map $z \mapsto g(z,b)-\zeta(z,b)$ for $z \in \Omega$. Recall $\Omega \subset \mathbb{R}^{n-1}$ is a cylinder of height $r$ whose cross sections are balls in $\mathbb{R}^{n-2}$ of radius $r$; hence $\mathcal{H}^{n-2}(\partial \Omega)=2\omega_{n-2}r^{n-2}+(n-2)w_{n-2}r^{n-2}$.

We estimate,
\begin{align*}
& \left\| \int_{Q} (g_{n-1}'-\zeta_{n-1}') \dd (\mathcal{L}^{n-1} \times \mu) \right\|_{H} \\
& \qquad \leq \int_{(-tr,tr)} \left\| \int_{\Omega} (g_{n-1}'-\zeta_{n-1}') \dd \mathcal{L}^{n-1} \right\|_{H} \dd \mu\\
& \qquad \leq \int_{(-tr,tr)} \frac{\tau r}{n(1+\|R_{\infty}\|_{H})} \mathcal{H}^{n-2}(\partial \Omega) \dd \mu\\
& \qquad = \frac{\mu(-tr,tr) \tau r (2\omega_{n-2}r^{n-2}+(n-2)\omega_{n-2}r^{n-2})}{n(1+\|R_{\infty}\|_{H})}\\
& \qquad =\frac{ \tau (\mathcal{L}^{n-1} \times \mu) (Q) }{(1+\|R_{\infty}\|_{H})}.
\end{align*}
Note that the function $\widetilde{\zeta}$ depends only on the projection onto $\mathbb{R}e$. Define $\Lambda\colon \mathbb{R} \to \mathbb{R}^{n-1}$ by $\Lambda(\theta)=\widetilde{\zeta}(\theta e)$ so that $\widetilde{\zeta}(u)=\Lambda(\langle u, e \rangle)$. It follows from Rademacher's theorem that $\Lambda'(\theta)$ exists for almost every $\theta$. Hence, using Fubini's theorem and the chain rule, we see $\widetilde{\zeta}_{n-1}'(u)=\Lambda'(\langle u, e \rangle) \otimes e$ for $\mathcal{L}^{n-1} \times \mu$ almost every $u$. Let $S_{0}=\{u \in \overline{\Omega}: \langle u,e\rangle =0\}$. Then, by Fubini's theorem,
\begin{align*}
\left\| \int_{Q} \widetilde{\zeta}_{n-1}' \dd (\mathcal{L}^{n-1} \times \mu) \right\|_{H} &\leq \int_{B(0,tr)} \left\| \int_{\Omega} \Lambda'(\langle u, e \rangle) \otimes e \dd \mathcal{L}^{n-1} \right\|_{H} \dd\mu\\
&\leq \mu(B(0,tr)) \left\| \int_{S_{0}} (\Lambda(0)-\Lambda(-r)) \otimes e  \dd \mathcal{H}^{n-2} \right\|_{H}\\
&= |\widetilde{\zeta}(-re)| \mathcal{H}^{n-2}(S_{0})\mu(B(0,tr))\\
&= |g(0)-g(-re)-rw|\mathcal{H}^{n-2}(S_{0})\mu(B(0,tr)).
\end{align*}
By using the expressions for $g$ and $w$ one can check,
\begin{align*}
&g(0)-g(-re)-rw\\
&\qquad  = - [f(x_{\infty}+T_{\infty}(-re))-f(x_{\infty})-L_{\infty}(-re)]\\
&\qquad \qquad +[f(x_{\infty}-T_{\infty}re)-f(x_{\infty}-\widetilde{T}_{\infty}re+\widetilde{x}-\kappa r \langle e,\widetilde{e}\rangle \widetilde{x}/|x|)]\\
& \qquad \qquad \qquad +[f(x_{\infty}+\widetilde{x})-f(x_{\infty}+x)].
\end{align*}
For the first term, as $|r|<\delta_{1}$, we may use \eqref{xiregdiff} to bound it by $\xi r$. Since $\kappa r=|x|$, the other terms can be made small, relative to $\xi r$, by choosing $\alpha$ small. Hence we obtain, by choice of $\xi$,
\begin{align*}
\left\| \int_{Q} \widetilde{\zeta}_{n-1}' \dd (\mathcal{L}^{n-1} \times \mu) \right\|_{H} &\leq 2\xi r \mathcal{H}^{n-2}(S_{0})\mu(B(0,tr))\\
&\leq 2\xi \mathcal{L}^{n-1}(\Omega)\mu(B(0,tr))\\ 
&\leq \frac{\tau (\mathcal{L}^{n-1}\times \mu)(Q)}{2(1+\|R_{\infty}\|_{H})}.
\end{align*}
Note that $\widetilde{\zeta}_{n-1}'=\zeta_{n-1}'-w \otimes e$. Hence we can estimate,
\begin{align*}
& \|R_{\infty}\|_{H} \left\| \dashint_{Q} (g_{n-1}'-w \otimes e) \dd(\mathcal{L}^{n-1} \times \mu) \right\|_{H}\\
& \qquad \leq \|R_{\infty}\|_{H} \left\| \dashint_{Q} (g_{n-1}'-\zeta_{n-1}') \dd(\mathcal{L}^{n-1} \times \mu) \right\|_{H}\\
& \qquad \qquad + \|R_{\infty}\|_{H}\left\| \dashint_{Q} (\zeta_{n-1}'-w \otimes e) \dd(\mathcal{L}^{n-1} \times \mu) \right\|_{H}\\
& \qquad \leq \tau + \frac{\tau}{2} < \frac{7\tau}{4}.
\end{align*}
This completes the proof of \eqref{positive} which proves Proposition \ref{finalprop} and hence Theorem \ref{theoremdiffinnullset}.

\begin{remark}\label{finrem}
When considering mappings from $\mathbb{R}^{n}$ to $\mathbb{R}^{m}$, $m<n$,
we may replace in the above arguments the measure $\mathcal{L}^{n-1}\times \mu$
with $\mathcal{L}^{m}\times \nu$ where $\nu$ is a doubling measure on $\mathbb{R}^{n-m}$.
By \citet{Wu98}, such $\mu$ can be chosen of arbitrarily small Hausdorff dimension. Together with the easy fact that any set is contained in a $G_{\delta}$ set of the same Hausdorff dimension, this leads to a set of Hausdorff dimension arbitrarily close to $m$ that contains points of differentiability of every Lipschitz function $f\colon \mathbb{R}^{n}\to\mathbb{R}^{m}$.
\end{remark}

\bibliographystyle{plainnat}

\bibliography{Differentiability}

\end{document}